\documentclass[12pt]{amsart}

\usepackage{amsmath}
\usepackage{amsfonts}
\usepackage{amssymb}
\usepackage[all]{xy}           %xypic macro for latex2.09

\usepackage{bbding}
\usepackage{txfonts}
\usepackage{amscd}
\usepackage{dsfont}%ʹÓÃ\mathds×ÖÌå
\usepackage{mathrsfs}

\usepackage{yfonts}

\usepackage[shortlabels]{enumitem}
\usepackage{ifpdf}
\ifpdf
  \usepackage[colorlinks,final,backref=page,hyperindex]{hyperref}
\else
  \usepackage[colorlinks,final,backref=page,hyperindex,hypertex]{hyperref}
\fi
\usepackage{tikz}
\usepackage[active]{srcltx}

\makeatletter

\newtheorem{thm}{Theorem}[section]
\newtheorem{lem}[thm]{Lemma}
\newtheorem{cor}[thm]{Corollary}
\newtheorem{pro}[thm]{Proposition}
\newtheorem{ex}[thm]{Example}
\newtheorem{rmk}[thm]{Remark}
\newtheorem{defi}[thm]{Definition}

\newcommand {\emptycomment}[1]{}

\newcommand{\lon }{\,\rightarrow\,}
\newcommand{\be }{\begin{equation}}
\newcommand{\ee }{\end{equation}}

\newcommand{\g}{\mathfrak g}

\newcommand{\huaB}{\mathcal{B}}%{{\mathcal{E}}}%{\mathcal{B}}

\newcommand{\huaL}{\mathcal{L}}

\newcommand{\huaE}{\mathcal{E}}

\newcommand{\huaV}{\mathcal{V}}

\newcommand{\huaH}{\mathcal{H}}

\newcommand{\Id}{\rm{Id}}

\newcommand{\br}[1]{   [ \cdot,    \cdot  ]   }

\newcommand{\Hom}{\mathrm{Hom}}

\newcommand{\gl}{\mathfrak {gl}}

\newcommand{\ad}{\mathrm{ad}}

\newcommand{\K}{\mathbb{K}}

\begin{document}

\title[Symplectic  structure, product and complex structures on Leibniz algebras]{Symplectic  structure, product structures and complex structures on Leibniz algebras}

\author{Rong Tang}
\address{Department of Mathematics, Jilin University, Changchun 130012, Jilin, China}
\email{tangrong@jlu.edu.cn}

\author{Nanyan Xu}
\address{Department of Mathematics, Jilin University, Changchun 130012, Jilin, China}
\email{xuny20@jlu.edu.cn}

\author{Yunhe Sheng}
\address{Department of Mathematics, Jilin University, Changchun 130012, Jilin, China}
\email{shengyh@jlu.edu.cn}

%\date{\today}

\begin{abstract}

In this paper, a symplectic structure on a Leibniz algebra is defined to be a {\em symmetric} nondegenerate bilinear form satisfying certain compatibility condition, and a phase space of a Leibniz algebra is defined to be a symplectic Leibniz algebra satisfying certain conditions. We show that a Leibniz algebra has a phase space if and only if there is a compatible Leibniz-dendriform algebra, and phase spaces of Leibniz algebras one-to-one corresponds to Manin triples of Leibniz-dendriform
algebras. Product (paracomplex) structures and complex structures on Leibniz algebras are studied in terms of decompositions of Leibniz algebras. A  para-K\"{a}hler structure on a Leibniz algebra is defined to be a symplectic structure and a paracomplex structure satisfying a compatibility condition. We show that a symplectic Leibniz algebra admits a para-K\"{a}hler structure if and only if the Leibniz algebra is the direct sum of two isotropic subalgebras as vector spaces. A complex product structure on a Leibniz algebra consists of a complex structure and a product structure satisfying a compatibility condition.  A  pseudo-K\"{a}hler structure on a Leibniz algebra is defined to be a symplectic structure and a complex structure satisfying a compatibility condition. Various properties and relations of complex product structures and pseudo-K\"{a}hler structures are studied. In particular, Leibniz-dendriform algebras give rise to complex product structures and pseudo-K\"{a}hler structures naturally.
\end{abstract}

\renewcommand{\thefootnote}{}

\maketitle

\tableofcontents

\allowdisplaybreaks

%\end{document}

\section{Introduction}
\label{sec:intr}

The purpose of this paper is to study symplectic structures, product structures, complex structures and related structures on Leibniz algebras. The main innovation and difference from the case of Lie algebras is using {\em symmetric} bilinear form in the definition of a symplectic structure on a Leibniz algebra.

 Leibniz algebras were first discovered by Bloh who called them D-algebras  \cite{Bloh}. Then Loday rediscovered this algebraic structure and called them Leibniz algebras \cite{Loday,Loday and Pirashvili}.  Leibniz algebras have various applications in both mathematics and physics. In particular, Leibniz algebras are  the underlying structures of embedding tensors, and have applications in higher gauge theories  \cite{KS,SW}. Recently, Leibniz algebras  are studied from different aspects.
Global and local integrations of Leibniz algebras were studied in \cite{BW,Int1}; deformation quantization of Leibniz algebras was studied in \cite{DW}; semisimple Leibniz algebras were studied in \cite{GKO}; representations and cohomologies of Leibniz algebras were studied in \cite{faithful rep,Bena,Omirov,Wagemann,conformal rep,Loday and Pirashvili,Sheng-Liu}.

Recall that a symplectic structure on a Lie algebra $\g$ is a {\em skew-symmetric} nondegenerate 2-cocycle. The underlying structure of a symplectic Lie algebra is a quadratic pre-Lie algebra \cite{Chu}. Since the operad of Leibniz algebras is an anticyclic operad \cite{Chapoton}, so one should use a {\em symmetric} bilinear form to define a {\bf symplectic  Leibniz algebra} and use  a {\em skew-symmetric} bilinear form to define a quadratic Leibniz algebra. A {\bf symplectic  structure} on a Leibniz algebra $(\huaE,[\cdot,\cdot]_\huaE)$ is a nondegenerate
symmetric bilinear form $\huaB$ satisfying the following equality:
\begin{eqnarray}\label{symplectic}
\huaB(z,[x,y]_\huaE)=-\huaB(y,[x,z]_\huaE)+\huaB(x,[y,z]_\huaE)+\huaB(x,[z,y]_\huaE).
\end{eqnarray}
This is different from the case of Lie algebras, whose operad  is  cyclic. Using skew-symmetric quadratic Leibniz algebras, a bialgebra theory was developed in \cite{Tang-Sheng} with connections to the Leibniz Yang-Baxter equation. Symmetric symplectic structures on Leibniz algebras one-to-one correspondence to symmetric solutions of the Leibniz classical Yang-Baxter equation. The notion of a Leibniz-dendriform algebra was also introduced in \cite{Tang-Sheng} as the underlying structure of Rota-Baxter operators on Leibniz algebras. Leibniz-dendriform algebras play important roles in the study of Leibniz algebras, similar to the role that pre-Lie algebras play in the study of Lie algebras. This kind of algebraic structures was further studied in \cite{Das} under the terminology of pre-Leibniz algebras.  We show that there is a one-to-one correspondence between (symmetric) symplectic Leibniz algebras and quadratic Leibniz-dendriform algebras. Also using the (symmetric) symplectic Leibniz algebras, we introduce the notion of a phase space of a Leibniz algebra. We show that a Leibniz algebra has a phase space if and only if there is a compatible Leibniz-dendriform algebra, and there is a one-to-one correspondence between Manin triples of Leibniz-dendriform
algebras and phase spaces of Leibniz algebras.

The notion of a Nijenhuis operator on a Leibniz algebra was introduced in \cite{Carinena} in the study of deformations of Dirac structures in a Lie bialgebroid. A linear map $E:\huaE\to \huaE$ is called a Nijenhuis operator on a Leibniz algebra
 $(\huaE,[\cdot,\cdot]_{\huaE})$ if it satisfies the following integrability condition:
  $$
  [Ex,Ey]_\huaE=E([Ex,y]_\huaE+[x,Ey]_\huaE-E[x,y]_\huaE).
  $$
   A product structure on a Leibniz algebra $\huaE$ is a linear map $E:\huaE\rightarrow\huaE$ satisfying $E^2=\Id$ and the above integrability condition. The existence of product structures on a Leibniz algebra can be characterized by certain decompositions of the Leibniz algebra. Two special cases, namely strict product structures and abelian product structures are studied in detail with applications to decompositions of Leibniz algebras. A paracomplex structure on a Leibniz algebra is defined to be a product structure such that the eigenspaces   associated to the eigenvalues $\pm1$ have the same
dimension.

Similarly, a complex structure on a real Leibniz algebra $\huaE$ is a linear map $E:\huaE\rightarrow\huaE$ satisfying $E^2=-\Id$ and the above integrability condition. The existence of complex structures on a real Leibniz algebra can also be characterized by certain decompositions of the Leibniz algebra. Parallel to the case of product structures, there are also two special complex structures, which are called strict complex structures and abelian complex structures.

Finally we study para-K\"{a}hler structures, complex product structures and pseudo-K\"{a}hler structures on a Leibniz algebra.
A para-K\"{a}hler structure is defined to be a symplectic structure and a paracomplex structure satisfying a compatibility condition. We show that a symplectic Leibniz algebra admits a para-K\"{a}hler structure if and only if the Leibniz algebra is the direct sum of two isotropic subalgebras as vector spaces. A complex product structure on a Leibniz algebra consists of a complex structure and a product structure satisfying a compatibility condition.  A  pseudo-K\"{a}hler structure on a Leibniz algebra is defined to be a symplectic structure and a complex structure satisfying a compatibility condition. Various properties and relations of complex product structures and pseudo-K\"{a}hler structures are studied. In particular, Leibniz-dendriform algebras give rise to complex product structures and pseudo-K\"{a}hler structures naturally. See \cite{Andrada0,ABD,Bai-2,Banayadi0,Bajo,Adela,Benayadi,LUV-1,LUV-2} for more studies of these structures in the context of Lie algebras.

 The paper is organized as follows. In Section \ref{sec:pre}, we recall some basic notions of Leibniz algebras, Leibniz-dendriform algebras and relative Rota-Baxter operators on Leibniz algebras. In Section \ref{sec:sym}, we study symlectic structures on Leibniz algebras and phase spaces of Leibniz algebras in terms of Leibniz-dendriform algebras. In Section \ref{sec:pro}, we study product structures on Leibniz algebras with applications to decompositions of Leibniz algebras. In Section \ref{sec:com}, we study complex structures on a real Leibniz algebra with applications to decompositions of the complexification of the Leibniz algebra. In Section \ref{sec:sympro}, we study para-K\"{a}hler structures on Leibniz algebras, which consists of a symplectic structure and a paracomplex structure satisfying a compatibility condition. In Section \ref{sec:compro}, we study  complex product structures on  Leibniz algebras which consists of a complex structure and a product structure satisfying a compatibility condition. In Section \ref{sec:symcom}, we study pseudo-K\"{a}hler structures on a Leibniz algebra which consists of a symplectic structure and a complex structure satisfying a compatibility condition.

In this paper, we work over an algebraically closed field $\mathds{K}$ of characteristic 0 and all the vector spaces are over $\K$ and finite-dimensional.

\section{Leibniz algebras, Leibniz-dendriform algebras and relative Rota-Baxter operators}\label{sec:pre}

In this section, we recall Leibniz algebras, Leibniz-dendriform algebras and relative Rota-Baxter operators on Leibniz algebras.

\begin{defi}
A {\bf Leibniz algebra} is a vector space $\huaE$ together with a bilinear map $[\cdot,\cdot]_\huaE:\huaE\otimes \huaE\to\huaE$ such that
\begin{eqnarray}
\label{Leibniz}[x,[y,z]_\huaE]_\huaE=[[x,y]_\huaE,z]_\huaE+[y,[x,z]_\huaE]_\huaE,\quad\forall x,y,z\in\huaE.
\end{eqnarray}
\end{defi}

\begin{defi}
A {\bf representation} of a Leibniz algebra $(\huaE,[\cdot,\cdot]_{\huaE})$ is a vector space $\huaV$ equipped with two linear maps $l,r:\huaE\lon\gl(\huaV)$ such that the following equalities hold for all $x,y\in\huaE$:
\begin{eqnarray}
\label{rep-1}l_{[x,y]_{\huaE}}&=&[l_x,l_y],\\
\label{rep-2}r_{[x,y]_{\huaE}}&=&[l_x,r_y],\\
\label{rep-3}r_y\circ l_x&=&-r_y\circ r_x.
\end{eqnarray}
\end{defi}

Define the left multiplication $L:\huaE\longrightarrow\gl(\huaE)$ and the right multiplication $R:\huaE\longrightarrow\gl(\huaE)$ respectively by $L_xy=[x,y]_\huaE$ and $R_xy=[y,x]_\huaE$ for all $x,y\in \huaE$. Then $(\huaE;L,R)$ is a representation of $(\huaE,[\cdot,\cdot]_{\huaE})$, which we call the {\bf regular representation}.
Define two linear maps $L^*,R^*:\huaE\longrightarrow\gl(\huaE^*)$ with $x\longrightarrow L^*_x$ and $x\longrightarrow R^*_x$ respectively (for all $x\in\huaE$) by
\begin{eqnarray}
\langle L^*_{x}\xi,y\rangle=-\langle \xi,[x,y]_\huaE\rangle,\,\,\,\,\langle R^*_{x}\xi,y\rangle=-\langle \xi,[y,x]_\huaE\rangle,\,\,\forall x,y\in\huaE,\,\,\xi\in\huaE^*.
\end{eqnarray}

\begin{lem}\label{lem:semidirectproduct}\cite{Tang-Sheng}
Let $(\huaV;l,r)$ be a representation of the Leibniz algebra $(\huaE,[\cdot,\cdot]_{\huaE})$. Then there is a Leibniz algebra structure on $\huaE\oplus\huaV$ given by:
\begin{eqnarray}\label{semidirect product}
[x+\xi,y+\eta]_\ltimes=[x,y]_\huaE+l_x(\eta)+r_y(\xi),\,\,\,\,\forall x,y\in\huaE,\,\,\xi,\eta\in\huaV.
\end{eqnarray}
This Leibniz algebra is called the {\bf semidirect product} of $\huaE$ and $\huaV$, and denoted by $\huaE\ltimes_{l,r}\huaV$.
\end{lem}

\begin{lem}\cite{Tang-Sheng}\label{dual rep}
Let $(\huaV;l,r)$ be a representation of the Leibniz algebra $(\huaE,[\cdot,\cdot]_{\huaE})$. Then
$$(\huaV^*;l^*,-l^*-r^*)$$
is a representation of $(\huaE,[\cdot,\cdot]_{\huaE})$, which is called the {\bf dual representation} of $(\huaV;l,r)$.
\end{lem}

\emptycomment{
\begin{defi}
The Leibniz cohomology of $\huaE$ with coefficients in $(\huaV;l,r)$ is the cohomology of the cochain complex $C^k(\huaE,\huaV)=
\Hom(\otimes^k\huaE,\huaV),(k\ge 0)$ with the coboundary operator
$$\partial:C^k(\huaE,\huaV)\longrightarrow C^
{k+1}(\huaE,\huaV)$$
defined by
\begin{eqnarray}
(\partial f)(x_1,\cdots,x_{k+1})&=&\sum_{i=1}^{k}(-1)^{i+1}l(x_i)f(x_1,\cdots,\hat{x_i},\cdots,x_{k+1})+(-1)^{k+1}r(x_{k+1})f(x_1,\cdots,x_{k})\\
                      \nonumber&&+\sum_{1\le i<j\le k+1}(-1)^if(x_1,\cdots,\hat{x_i},\cdots,x_{j-1},[x_i,x_j]_\g,x_{j+1},\cdots,x_{k+1}).
\end{eqnarray}
The resulting cohomology is denoted by $\huaH^*(\huaE,\huaV)$.
\end{defi}

 An important representation is the regular representation $(\huaE,L,R)$. We denote the complex by $(C^*(\huaE,\huaE),\partial^{reg})$. The resulting cohomology is denoted by $\huaH^*_{\rm reg}$.}

\begin{defi}\cite{Tang-Sheng}
A {\bf Leibniz-dendriform algebra} is a vector space $\huaL$ equipped with two bilinear maps
$\lhd$ and $\rhd:\huaL\otimes \huaL\lon \huaL$ such that the following equalities hold for all
$x,y,z\in \huaL$:
\begin{eqnarray}
\label{p1}(x\lhd y)\lhd z&=&x\lhd(y\lhd z)-y\lhd(x\lhd z)-(x\rhd y)\lhd z,\\
\label{p2}x\lhd(y\rhd z) &=&(x\lhd y)\rhd z+y\rhd(x\lhd z)+y\rhd(x\rhd z),\\
\label{p3}x\rhd(y\rhd z) &=&(x\rhd y)\rhd z+y\lhd(x\rhd z)-x\rhd(y\lhd z).
\end{eqnarray}
\end{defi}

\begin{ex}\cite{Tang-Sheng}
Let $V$ be a vector space. Define two bilinear maps $\lhd$ and $\rhd$ on the vector space
$\gl(V)\oplus V$ as follows:
\begin{eqnarray*}
(A+u)\lhd(B+v)=AB+Av,\quad (A+u)\rhd(B+v)=-BA.
\end{eqnarray*}
Then $(\gl(V)\oplus V,\lhd,\rhd)$ is a Leibniz-dendriform algebra.
\end{ex}

\begin{pro}\cite{Tang-Sheng}\label{LeiDen rep}
Let $(\huaL,\lhd,\rhd)$ be a Leibniz-dendriform algebra. Then $(\huaL,[\cdot,\cdot]_{\lhd, \rhd})$ is a Leibniz algebra, which is called the {\bf sub-adjacent Leibniz algebra} of $(\huaL,\lhd,\rhd)$ and $(\huaL,\lhd,\rhd)$ is called a
{\bf compatible Leibniz-dendriform algebra} structure on $(\huaL,[\cdot,\cdot]_{\lhd, \rhd})$, where the bilinear map $[\cdot,\cdot]_{\lhd, \rhd}:\huaL\otimes \huaL\lon \huaL$ is given by
\begin{eqnarray}\label{LeiDen to Lei}
[x,y]_{\lhd, \rhd}=x\lhd y+x\rhd y,\,\,\,\,\forall x,y\in \huaL.
\end{eqnarray}

 Moreover,  $(\huaL;L_\lhd,R_\rhd)$ is a representation of the Leibniz algebra $(\huaL,[\cdot,\cdot]_{\lhd, \rhd})$, where $L_\lhd$ and $R_\rhd:\huaL\lon\gl(\huaL)$
are given by
\begin{eqnarray*}
L_{\lhd x}y=x\lhd y,\,\,\,\,R_{\rhd x}y=y\rhd x,\,\,\,\,\forall x,y\in \huaL.
\end{eqnarray*}
\end{pro}

 Finally we recall the notion of   relative Rota-Baxter operators on Leibniz algebras, which naturally give rise to Leibniz-dendriform algebras.

\begin{defi}\cite{Tang-Sheng}
Let $(\huaV;l,r)$ be a representation of the Leibniz algebra $(\huaE,[\cdot,\cdot]_\huaE)$. A linear map $T:\huaV\lon\huaE$ is called a {\bf relative Rota-Baxter operator} on $(\huaE,[\cdot,\cdot]_\huaE)$ with respect to $(\huaV;l,r)$ if $T$ satisfies:
\begin{eqnarray}
[Tv_1,Tv_2]_\huaE=T(l(Tv_1)v_2+r(Tv_2)v_1),\,\,\,\,\forall v_1,v_2\in\huaV.
\end{eqnarray}
\end{defi}

\begin{pro}\cite{Tang-Sheng}\label{Leibniz-dendriform}
Let $(\huaV;l,r)$ be a representation of the Leibniz algebra $(\huaE,[\cdot,\cdot]_\huaE)$ and $T:\huaV\lon\huaE$ a relative Rota-Baxter operator on $\huaE$ with respect to $(\huaV;l,r)$. Then there exists a Leibniz-dendriform algebra structure on $\huaV$ given by
$$ v_1\lhd v_2=l(Tv_1)v_2,\,\,\,\,v_1\rhd v_2=r(Tv_2)v_1,\,\,\,\,\forall v_1,v_2\in \huaV.$$
Moreover, $T$ is a Leibniz algebra homomorphism from $(\huaV,[\cdot,\cdot]_{\lhd, \rhd})$ to $(\huaE,[\cdot,\cdot]_\huaE)$.
\end{pro}

\begin{pro}\cite{Tang-Sheng}\label{o to Lei-Den}
Let $(\huaE,[\cdot,\cdot]_\huaE)$ be a Leibniz algebra. Then there is a compatible Leibniz-dendriform algebra on $\huaE$ if and only if there exists an invertible relative Rota-Baxter operator $T:\huaV\lon\huaE$ on $\huaE$ with respect to a representation $(\huaV;l,r)$.  Furthermore, the compatible Leibniz-dendriform algebra structure on $\huaE$ is given by
\begin{eqnarray*}
x\lhd y=T(l(x)T^{-1}y),\,\,\,\,x\rhd y=T(r(y)T^{-1}x),\,\,\,\,\forall x,y\in \huaE.
\end{eqnarray*}
\end{pro}

\section{Symplectic structures and phase spaces of Leibniz algebras}\label{sec:sym}

In this section, we introduce the notion of a phase space of a Leibniz algebra  using the symmetric symplectic structure and show that a Leibniz algebra has a phase space if and only if there is a compatible Leibniz-dendriform algebra. Moreover, we introduce the notion of a Manin triple of Leibniz-dendriform algebras and show that there is a one-to-one correspondence between Manin triples of Leibniz-dendriform algebras and phase spaces of Leibniz algebras.

\emptycomment{
\subsection{Symplectic Leibniz algebras and quadratic  Leibniz-dendriform algebras}

\begin{defi}
A {\bf symplectic  structure} on a Leibniz algebra $(\huaE,[\cdot,\cdot]_\huaE)$ is a nondegenerate
symmetric bilinear form $\huaB$ satisfying the following equality:
\begin{eqnarray}\label{symplectic}
\huaB(z,[x,y]_\huaE)=-\huaB(y,[x,z]_\huaE)+\huaB(x,[y,z]_\huaE)+\huaB(x,[z,y]_\huaE).
\end{eqnarray}
\end{defi}

\begin{rmk}
Let $\huaE$ be a Leibniz algebra and $r\in\huaE\otimes\huaE$ symmetric and nondegenerate. Then $r$ is a solution of the classical Leibniz-Yang-Baxter equation in $\huaE$ if and only if the bilinear form $\huaB$ on $\huaE$ given by
\begin{eqnarray}
\huaB(x,y):=\langle r^{-1}(x),y\rangle,~~\forall x,y\in\huaE,
\end{eqnarray}
is a symplectic structure on the Leibniz algebra $(\huaE,[\cdot,\cdot]_\huaE)$. More details are available in \cite{Tang-Sheng}.
\end{rmk}
}

\begin{defi}
A {\bf quadratic Leibniz-dendriform algebra} $(\huaL,\lhd,\rhd, \huaB)$ is a Leibniz-dendriform algebra $(\huaL,\lhd,\rhd)$ equipped with a nondegenerate symmetric bilinear form
$\huaB\in \huaL^*\otimes \huaL^*$ such that the following invariant conditions hold for
all $x,y\in\huaE$:
\begin{eqnarray}
\huaB(x\lhd y,z)&=&-\huaB(y,[x,z]_{\lhd,\rhd}),\label{Leiden to Lei_1}\\
\huaB(x\rhd y,z)&=&\huaB(x,[y,z]_{\lhd,\rhd})+\huaB(x,[z,y]_{\lhd,\rhd}).\label{Leiden to Lei_2}
\end{eqnarray}

\end{defi}

\begin{thm}\label{thm:qua-leiden-symlei}
  There is a one-to-one correspondence between symplectic Leibniz algebras and quadratic Leibniz-dendriform algebras.
\end{thm}

\begin{proof}
Let $(\huaL,\lhd,\rhd,\huaB)$ be a quadratic Leibniz-dendriform algebra. By \eqref{LeiDen to Lei}, \eqref{Leiden to Lei_1} and \eqref{Leiden to Lei_2}, we have
\begin{eqnarray*}
\huaB(z,[x,y]_{\lhd,\rhd})=-\huaB(y,[x,z]_{\lhd,\rhd})+\huaB(x,[y,z]_{\lhd,\rhd})+\huaB(x,[z,y]_{\lhd,\rhd}).
\end{eqnarray*}
Thus $(\huaL,[\cdot,\cdot]_{\lhd,\rhd},\huaB)$ is a symplectic Leibniz algebra.

Conversely, let $(\huaE,[\cdot,\cdot]_\huaE,\huaB)$ be a symplectic Leibniz algebra. Then there is a compatible quadratic Leibniz-dendriform algebra structure $(\huaE,\lhd,\rhd,\huaB)$ given by:
\begin{eqnarray}
\huaB(x\lhd y,z)&=&-\huaB(y,[x,z]_\huaE),\label{sym-to-LeiDen1}\\
\huaB(x\rhd y,z)&=&\huaB(x,[y,z]_\huaE)+\huaB(x,[z,y]_\huaE).\label{sym-to-LeiDen2}
\end{eqnarray}
In fact, define a linear map $T:\huaE^*\lon \huaE$ by $\huaB(T\xi,x)=\langle\xi,x\rangle$ for all $x\in \huaE$ and $\xi\in \huaE^*$. By \eqref{symplectic}, we obtain that $T$ is an invertible relative Rota-Baxter operator on $\huaE$ with respect to $(\huaE^*;L^*,-L^*-R^*)$. By Proposition \ref{o to Lei-Den}, there exists a compatible Leibniz-dendriform algebra on $\huaE$ given by
$$x\lhd y=T(L^*(x)T^{-1}y),\,\,\,\,x\rhd y=T((-L^*-R^*)(y)T^{-1}x),\,\,\,\,\forall x,y\in \huaE.$$
Then for all $x,y,z\in \huaE$, we have
\begin{eqnarray*}
\huaB(x\lhd y,z)&=&\huaB(T(L^*(x)T^{-1}y),z)=\langle L^*(x)T^{-1}y,z\rangle\\
&=&-\langle T^{-1}y,[x,z]_\huaE\rangle=-\huaB(y,[x,z]_\huaE),\\
\huaB(x\rhd y,z)&=&\huaB(T((-L^*-R^*)(y)T^{-1}x),z)=\langle-L^*(y)T^{-1}x,z\rangle+\langle-R^*(y)T^{-1}x,z\rangle\\
&=&\langle T^{-1}x,[y,z]_\huaE\rangle+\langle T^{-1}x,[z,y]_\huaE\rangle=\huaB(x,[y,z]_\huaE)+\huaB(x,[z,y]_\huaE),
\end{eqnarray*}
 which implies that $(\huaE,\lhd,\rhd,\huaB)$ is a quadratic Leibniz-dendriform algebra.
\end{proof}

%\subsection{Phase spaces of Leibniz  algebras}

Let $V$ be a vector space and $V^{*}=$ Hom $(V,\mathds{K})$ its dual space. Then there is a natural nondegenerate symmetric bilinear form $\huaB$ on
$T^{*}V=V\oplus V^{*}$ given by:
\begin{eqnarray}\label{phase space symplectic}
\huaB(x+\xi,y+\eta)=\langle x,\eta\rangle+\langle\xi,y\rangle,\,\,\,\forall x,y\in V,\,\,\xi,\eta\in V^*.
\end{eqnarray}

\begin{defi}
Let $(\huaE,[\cdot,\cdot]_\huaE)$ be a Leibniz algebra and $\huaE^*$ its dual space. If there is a Leibniz algebra structure $[\cdot,\cdot]$ on the direct sum vector space $T^*\huaE=\huaE\oplus\huaE^*$ such that $(\huaE\oplus\huaE^*,[\cdot,\cdot],\huaB)$ is a symplectic Leibniz algebra, where $\huaB$ given by \eqref{phase space symplectic}, and both $(\huaE,[\cdot,\cdot]_\huaE)$ and $(\huaE^*,[\cdot,\cdot]\mid_{\huaE^*})$ are Leibniz subalgebras of $(\huaE\oplus\huaE^*,[\cdot,\cdot])$, then the symplectic Leibniz algebra $(\huaE\oplus\huaE^*,[\cdot,\cdot],\huaB)$ is called a {\bf phase space} of the Leibniz algebra $(\huaE,[\cdot,\cdot]_\huaE)$.
\end{defi}

\begin{thm}\label{phase space and Lei-Den}
A Leibniz algebra has a phase space if and only if there is a compatible Leibniz-dendriform algebra.
\end{thm}

\begin{proof}
Let $(\huaL,\lhd,\rhd)$ be a Leibniz-dendriform algebra. By Proposition \ref{LeiDen rep}, $(\huaL;L_\lhd,R_\rhd)$ is a representation of the sub-adjacent Leibniz algebra $(\huaL,[\cdot,\cdot]_{\lhd,\rhd})$. By Lemma \ref{dual rep}, $(\huaL^*;L^*_\lhd,-L^*_\lhd-R^*_\rhd)$ is a representation of $(\huaL,[\cdot,\cdot]_{\lhd,\rhd})$. Thus we have the semidirect product Leibniz algebra $\huaL\ltimes_{L^*_\lhd,-L^*_\lhd-R^*_\rhd}\huaL^*$ by Lemma \ref{lem:semidirectproduct}. For all $x,y,z\in\huaL$ and $\xi,\eta,\zeta\in\huaL^*$, we have
\begin{eqnarray*}
&&\huaB(z+\zeta,[x+\xi,y+\eta]_\ltimes)\\
&=&\huaB(z+\zeta,[x,y]_{\lhd,\rhd}+L^*_{\lhd x}\eta+(-L^*_{\lhd y}-R^*_{\rhd y})\xi)\\
&=&\langle z,L^*_{\lhd x}\eta+(-L^*_{\lhd y}-R^*_{\rhd y})\xi\rangle+\langle\zeta,[x,y]_{\lhd,\rhd}\rangle\\
&=&-\langle x\lhd z,\eta\rangle+\langle y\lhd z,\xi\rangle+\langle z\rhd y,\xi\rangle+\langle\zeta,x\lhd y\rangle+\langle\zeta,x\rhd y\rangle,\\
&&-\huaB(y+\eta,[x+\xi,z+\zeta]_\ltimes)\\
&=&-\huaB(y+\eta,[x,z]_{\lhd,\rhd}+L^*_{\lhd x}\zeta+(-L^*_{\lhd z}-R^*_{\rhd z})\xi)\\
&=&-\langle y,L^*_{\lhd x}\zeta+(-L^*_{\lhd z}-R^*_{\rhd z})\xi\rangle-\langle\eta,[x,z]_{\lhd,\rhd}\rangle\\
&=&\langle x\lhd y,\zeta\rangle-\langle z\lhd y,\xi\rangle-\langle y\rhd z,\xi\rangle-\langle\eta,x\lhd z\rangle-\langle\eta,x\rhd z\rangle,\\
&&\huaB(x+\xi,[y+\eta,z+\zeta]_\ltimes)\\
&=&\huaB(x+\xi,[y,z]_{\lhd,\rhd}+L^*_{\lhd y}\zeta+(-L^*_{\lhd z}-R^*_{\rhd z})\eta)\\
&=&\langle x,L^*_{\lhd y}\zeta+(-L^*_{\lhd z}-R^*_{\rhd z})\eta\rangle+\langle\xi,[y,z]_{\lhd,\rhd}\rangle\\
&=&-\langle y\lhd x,\zeta\rangle+\langle z\lhd x,\eta\rangle+\langle x\rhd z,\eta\rangle+\langle\xi,y\lhd z\rangle+\langle\xi,y\rhd z\rangle,\\
&&\huaB(x+\xi,[z+\zeta,y+\eta]_\ltimes)\\
&=&\huaB(x+\xi,[z,y]_{\lhd,\rhd}+L^*_{\lhd z}\eta+(-L^*_{\lhd y}-R^*_{\rhd y})\zeta)\\
&=&\langle x,L^*_{\lhd z}\eta+(-L^*_{\lhd y}-R^*_{\rhd y})\zeta\rangle+\langle\xi,[z,y]_{\lhd,\rhd}\rangle\\
&=&-\langle z\lhd x,\eta\rangle+\langle y\lhd x,\zeta\rangle+\langle x\rhd y,\zeta\rangle+\langle\xi,z\lhd y\rangle+\langle\xi,z\rhd y\rangle.
\end{eqnarray*}
Therefore $\huaB$ satisfies \eqref{symplectic} which means that $(\huaL\oplus\huaL^*,[\cdot,\cdot]_\ltimes,\huaB)$ is a symplectic Leibniz algebra.
Moreover, $(\huaL,[\cdot,\cdot]_\huaL)$ is a subalgebra of $\huaL\ltimes_{L^*_\lhd,-L^*_\lhd-R^*_\rhd}\huaL^*$ and $(\huaL^*,[\cdot,\cdot]_\ltimes\mid_{\huaL^*})$ is an abelian subalgebra of $\huaL\ltimes_{L^*_\lhd,-L^*_\lhd-R^*_\rhd}\huaL^*$. Thus, the symplectic Leibniz algebra $(\huaL\oplus\huaL^*,[\cdot,\cdot]_\ltimes,\huaB)$ is a phase space of the sub-adjacent Leibniz algebra $(\huaL,[\cdot,\cdot]_{\lhd,\rhd})$.

Conversely, let $(T^*\huaE=\huaE\oplus\huaE^*,[\cdot,\cdot],\huaB)$ be a phase space of the Leibniz algebra $(\huaE,[\cdot,\cdot]_\huaE)$. By Theorem \ref{thm:qua-leiden-symlei}, there exists a compatible Leibniz-dendriform algebra structure on $T^*\huaE$ which is given by:
\begin{eqnarray*}
\huaB((x+\xi)\lhd (y+\eta),z+\zeta)&=&-\huaB(y+\eta,[x+\xi,z+\zeta]),\\
\huaB((x+\xi)\rhd (y+\eta),z+\zeta)&=&\huaB(x+\xi,[y+\eta,z+\zeta])+\huaB(x+\xi,[z+\zeta,y+\eta]).
\end{eqnarray*}
Since $(\huaE,[\cdot,\cdot]_\huaE)$ is a subalgebra of $(T^*\huaE,[\cdot,\cdot])$, we have
\begin{eqnarray*}
\huaB(x\lhd y,z)&=&-\huaB(y,[x,z])=-\huaB(y,[x,z]_\huaE)=-\langle y,0\rangle-\langle0,[x,z]_\huaE\rangle=0,\\
\huaB(x\rhd y,z)&=&\huaB(x,[y,z])+\huaB(x,[z,y])=\huaB(x,[y,z]_\huaE)+\huaB(x,[z,y]_\huaE)=0,
\end{eqnarray*}
for all $x,y,z\in\huaE$. Thus, $x\lhd y, x\rhd y\in\huaE$, which implies that $(\huaE,\lhd{\mid_{\huaE}},\rhd{\mid_\huaE})$ is a subalgebra of the Leibniz-dendriform algebra$(T^*\huaE,\lhd,\rhd)$. Its sub-adjacent Leibniz algebra $(\huaE,[\cdot,\cdot]_{\lhd\mid_{\huaE},\rhd{\mid_\huaE}})$ is exactly the original Leibniz algebra $(\huaE,[\cdot,\cdot]_\huaE)$.
\end{proof}

\begin{cor}\label{phase to LeiDen}
Let $(T^*\huaE=\huaE\oplus\huaE^*,[\cdot,\cdot],\huaB)$ be a phase space of the Leibniz algebra $(\huaE,[\cdot,\cdot]_\huaE)$ and $(\huaE\oplus\huaE^*,\lhd,\rhd)$ the compatible Leibniz-dendriform algebra. Then both $(\huaE,\lhd{\mid_{\huaE}},\rhd{\mid_\huaE})$ and $(\huaE^*,\lhd{\mid_{\huaE^*}},\rhd{\mid_{\huaE^*}})$ are subalgebras of the Leibniz-dendriform algebra $(\huaE\oplus\huaE^*,\lhd,\rhd)$.
\end{cor}

\begin{cor}
Let $(\huaE;l,r)$ be a representation of the Leibniz algebra $(\huaE,[\cdot,\cdot]_\huaE)$ and $(\huaE^*;l^*,-l^*-r^*)$ the dual representation. If $(\huaE\oplus\huaE^*,[\cdot,\cdot]_\ltimes,\huaB)$ is a phase space of $(\huaE,[\cdot,\cdot]_\huaE)$, where $[\cdot,\cdot]_\ltimes$ is given by \eqref{semidirect product}, then
\begin{eqnarray*}
x\lhd y=l_x y,\,\,\,\,x\rhd y=r_y x,\,\,\,\,\forall x,y\in\huaE,
\end{eqnarray*}
defines a Leibniz-dendriform algebra structure on $\huaE$.
\end{cor}

\begin{proof}
By \eqref{sym-to-LeiDen1}, \eqref{sym-to-LeiDen2} and \eqref{phase space symplectic}, for all $x,y\in\huaE$ and $\xi\in\huaE^*$, we have
\begin{eqnarray*}
\langle x\lhd y,\xi\rangle&=&\huaB(x\lhd y,\xi)=-\huaB(y,[x,\xi]_\ltimes)=-\huaB(y,l^*_x \xi)=-\langle y,l^*_x \xi\rangle=\langle l_x y,\xi\rangle,\\
\langle x\rhd y,\xi\rangle&=&\huaB(x\rhd y,\xi)=\huaB(x,[y,\xi]_\ltimes)+\huaB(x,[\xi,y]_\ltimes)=\huaB(x,l^*_y \xi)+\huaB(x,(-l^*_y-r^*_y)\xi)\\
&=&\langle x,l^*_y \xi\rangle-\langle x,l^*_y \xi\rangle-\langle x,r^*_y \xi\rangle=\langle r_y x,\xi\rangle.
\end{eqnarray*}
Therefore, $x\lhd y=l_x y$ and $x\rhd y=r_y x$.
\end{proof}

\begin{defi}
A {\bf Manin triple of Leibniz-dendriform algebras} is a triple $(\mathfrak{L};\huaL,\huaL')$, where
\begin{itemize}
  \item $(\mathfrak{L},\lhd,\rhd,\huaB)$ is a quadratic Leibniz-dendriform algebra;
  \item both $\huaL$ and $\huaL'$ are isotropic subalgebras of $(\mathfrak{L},\lhd,\rhd)$, that is, $\huaB(x_1,y_1)=\huaB(x_2,y_2)=0$ for $x_1,y_1\in\huaL$ and $x_2,y_2\in\huaL'$;
  \item $\mathfrak{L}=\huaL\oplus\huaL'$ as vector spaces.
\end{itemize}
\end{defi}

In a Manin triple $(\mathfrak{L};\huaL,\huaL')$ of  Leibniz-dendriform algebras, since the symmetric bilinear form $\huaB$ is nondegenerate, $\huaL'$ can be identified with $\huaL^*$ via
$$\langle\xi,x\rangle\triangleq\huaB(\xi,x),\,\,\,\,\forall x\in\huaL,\,\,\xi\in\huaL'.$$
Thus, $\mathfrak{L}$ is isomorphic to $\huaL\oplus\huaL^*$ naturally and the bilinear form $\huaB$ is exactly given by \eqref{phase space symplectic}. %By \eqref{Leiden to Lei_1} and \eqref{Leiden to Lei_2},   the precise form of the Leibniz-dendriform structure $\lhd,\rhd$ on $\huaL\oplus\huaL^*$ are given by:
%\begin{eqnarray*}
 % (x+\xi)\lhd(y+\eta)&=&\\
  %(x+\xi)\rhd(y+\eta)&=&
%\end{eqnarray*}

\begin{thm}
There is a one-to-one correspondence between Manin triples of Leibniz-dendriform algebras and phase spaces of Leibniz algebras.
%More precisely, if $(\huaL\oplus\huaL^*;\huaL,\huaL^*)$ is a Manin triple of Leibniz-dendriform algebras, then $(\huaL\oplus\huaL^*,[\cdot,\cdot]_{\lhd,\rhd},\huaB)$ is a symplectic Leibniz algebra, where $\huaB$ is given by Eq.(\ref{phase space symplectic}). Conversely, if $(\huaE\oplus\huaE^*,[\cdot,\cdot],\huaB)$ is a phase space of Leibniz algebra $(\huaE,[\cdot,\cdot]_{\huaE})$, then $(\huaE\oplus\huaE^*;\huaE,\huaE^*)$ is a Manin triple of Leibniz-dendriform algebras, where the Leibniz-dendriform algebra structure on $\huaE\oplus\huaE^*$ is given by Eq.(\ref{sym-to-LeiDen1}) and (\ref{sym-to-LeiDen2}).
\end{thm}

\begin{proof}
Let $(\huaL\oplus\huaL^*;\huaL,\huaL^*)$ be a Manin triple of Leibniz-dendriform algebras. Here $(\huaL\oplus\huaL^*,\lhd,\rhd,\huaB)$ is a quadratic Leibniz-dendriform algebra and $\huaB$ is given by \eqref{phase space symplectic}. Denote by $(\huaL,\lhd{\mid_\huaL},\rhd{\mid_\huaL})$ and $(\huaL^*,\lhd{\mid_{\huaL^*}},\rhd{\mid_{\huaL^*}})$ the subalgebras of $(\huaL\oplus\huaL^*,\lhd,\rhd)$, and denote by $(\huaL,[\cdot,\cdot]_{\lhd{\mid_\huaL},\rhd{\mid_\huaL}})$ and $(\huaL^*,[\cdot,\cdot]_{\lhd{\mid_{\huaL^*}},\rhd{\mid_{\huaL^*}}})$ the corresponding sub-adjacent Leibniz algebras, which are obviously Leibniz subalgebras of $(\huaL\oplus\huaL^*,[\cdot,\cdot]_{\lhd,\rhd})$.
By Theorem \ref{thm:qua-leiden-symlei}, $(\huaL\oplus\huaL^*,[\cdot,\cdot]_{\lhd,\rhd},\huaB)$ is a symplectic Leibniz algebra, which is a phase space of the Leibniz algebra $(\huaL,[\cdot,\cdot]_{\lhd{\mid_\huaL},\rhd{\mid_\huaL}})$.

Conversely, let $(\huaE\oplus\huaE^*,[\cdot,\cdot],\huaB)$ be a phase space of the Leibniz algebra $(\huaE,[\cdot,\cdot]_\huaE)$. By Theorem \ref{thm:qua-leiden-symlei}, there exists a compatible Leibniz-dendriform algebra structure $\lhd,\rhd$ on $\huaE\oplus\huaE^*$ given by \eqref{sym-to-LeiDen1} and \eqref{sym-to-LeiDen2} such that $(\huaE\oplus\huaE^*,\lhd,\rhd,\huaB)$ is a quadratic Leibniz-dendriform algebra. By Corollary \ref{phase to LeiDen}, $(\huaE,\lhd{\mid_{\huaE}},\rhd{\mid_{\huaE}})$ and $(\huaE^*,\lhd{\mid_{\huaE^*}},\rhd{\mid_{\huaE^*}})$ are Leibniz-dendriform subalgebras of $(\huaE\oplus\huaE^*,\lhd,\rhd)$. Since $\huaB$ is given by \eqref{phase space symplectic}, then $\huaE$ and $\huaE^*$ are isotropic clearly. Thus $(\huaE\oplus\huaE^*;\huaE,\huaE^*)$ is a Manin triple of Leibniz-dendriform algebras.
\end{proof}

\begin{ex}\label{sym example 1}
Consider the $4$-dimensional Leibniz algebra $({\huaE},[\cdot,\cdot]_{\huaE})$ given with respect to the basis $\{e_1,e_2,e_3,e_4\}$  by:
\begin{eqnarray*}
[e_1,e_3]_{\huaE}=2e_4.
\end{eqnarray*}
Then $\huaB\in{\huaE}^*\otimes{\huaE}^*$
\emptycomment{
\begin{eqnarray*}
\huaB=\sum_{i,j=1}^4 b_{ij}{e_i}^*\otimes {e_j}^*,\,\,\,b_{ij}=b_{ji}\in\mathds{K},
\end{eqnarray*}}
is a symplectic structure if and only if the matrix of $\huaB$ with respect to the above basis is of the following form:
\begin{center}
$\begin{pmatrix}
* & * & * & *\\
* & * & * & 0\\
* & * & * & 0\\
* & 0 & 0 & 0
\end{pmatrix}$,
\end{center}
where the matrix is nondegenerate and symmetric.
\emptycomment{
$b_{24}=b_{42}=b_{34}=b_{43}=b_{44}=0$, where $\{{e_1}^*,{e_2}^*,{e_3}^*,{e_4}^*\}$ is the dual basis.}
\end{ex}

\begin{ex}
The special linear Lie algebra $\mathfrak{sl_n}(\mathds{C})$ with its Killing form, i.e. $\huaB(x,y)=$tr$(\ad_x\circ\ad_y)$ is a symplectic Leibniz algebra.
\end{ex}

\begin{ex}
Let $T:\huaV\rightarrow\huaE$ be a relative Rota-Baxter operator on $\huaE$ with respect to the representation $(\huaV;l,r)$. By Proposition \ref{Leibniz-dendriform}, we obtain a Leibniz-dendriform algebra $(\huaV,\lhd,\rhd)$. By Theorem \ref{phase space and Lei-Den}, $(\huaV\oplus\huaV^*,[\cdot,\cdot]_\ltimes,\huaB)$ is a symplectic Leibniz algebra, where $\huaB$ is given by \eqref{phase space symplectic}, $[\cdot,\cdot]_\ltimes$ is given by
\begin{eqnarray*}
[u+\alpha,v+\beta]_\ltimes=l(Tu)v+r(Tv)u+l^*(Tu)\beta-l^*(Tv)\alpha-r^*(Tv)\alpha,\,\,\,\forall u,v\in\huaV,\alpha,\beta\in\huaV^*.
\end{eqnarray*}
\end{ex}

\section{Product structures on Leibniz algebras}\label{sec:pro}

In this section, we introduce the notion of a product structure on a Leibniz algebra using the Nijenhuis condition as the integrability condition. We find two special product structures, each of them gives a special decomposition of the original Leibniz algebra.

\begin{defi}\label{def of product}
Let $(\huaE,[\cdot,\cdot]_{\huaE})$ be a Leibniz algebra. An {\bf almost product structure} on $\huaE$ is a linear endomorphism $E:\huaE\rightarrow\huaE$ satisfying $E^2=\Id$. A {\bf product structure} on $\huaE$ is an almost product structure $E$ satisfying the following integrability condition:
\begin{eqnarray}\label{product str}
E[x,y]_{\huaE}=[Ex,y]_{\huaE}+[x,Ey]_{\huaE}-E[Ex,Ey]_{\huaE},\,\,\,\forall x,y\in\huaE.
\end{eqnarray}
\end{defi}

\begin{rmk}
One can understand a product structure on a Leibniz algebra as a Nijenhuis operator $E$ satisfying $E^2=\Id$. See \cite{Carinena} for more details about Nijenhuis operators on Leibniz algebras.
\end{rmk}

\begin{thm}\label{product to decomposition}
Let $(\huaE,[\cdot,\cdot]_{\huaE})$ be a Leibniz algebra. Then $\huaE$ has a product structure if and only if $\huaE$ admits a decomposition:
\begin{eqnarray}
\huaE=\huaE_+\oplus\huaE_-,
\end{eqnarray}
where $\huaE_+$ and $\huaE_-$ are subalgebras of $\huaE$.
\end{thm}

\begin{proof}
Let $E$ be a product structure on $\huaE$. By $E^2=\Id$, we have $\huaE=\huaE_+\oplus\huaE_-$, where $\huaE_+$ and $\huaE_-$ are the eigenspaces of $\huaE$ associated to the eigenvalues $\pm1$. Since
\begin{eqnarray*}
E[x_1,x_2]_{\huaE}&=&[Ex_1,x_2]_{\huaE}+[x_1,,Ex_2]_{\huaE}-E[Ex_1,Ex_2]_{\huaE}=[x_1,x_2]_{\huaE}+[x_1,x_2]_{\huaE}-E[x_1,x_2]_{\huaE}\\
&=&2[x_1,x_2]_{\huaE}-E[x_1,x_2]_{\huaE},
\end{eqnarray*}
for all $x_1,x_2\in\huaE_+$, we have $[x_1,x_2]_{\huaE}\in\huaE_+$, which implies that $\huaE_+$ is a subalgebra. In the same way, we can show that $\huaE_-$ is a subalgebra.

Conversely,  define a linear endomorphism $E:\huaE\rightarrow\huaE$ by
\begin{eqnarray}\label{decomposition to product}
E(x+\xi)=x-\xi,\,\,\,\forall x\in\huaE_+,\,\xi\in\huaE_-.
\end{eqnarray}
Obviously $E^2=\Id$. Since $\huaE_+$ is a subalgebra of $\huaE$, for all $x_1,x_2\in\huaE_+$, we have
\begin{eqnarray*}
[Ex_1,x_2]_{\huaE}+[x_1,Ex_2]_{\huaE}-E[Ex_1,Ex_2]_{\huaE}&=&[x_1,x_2]_{\huaE}+[x_1,x_2]_{\huaE}-E[x_1,x_2]_{\huaE}\\
&=&[x_1,x_2]_{\huaE}+[x_1,x_2]_{\huaE}-[x_1,x_2]_{\huaE}=[x_1,x_2]_{\huaE}\\
&=&E[x_1,x_2]_{\huaE},
\end{eqnarray*}
which implies that (\ref{product str}) holds for all $x_1,x_2\in\huaE_+$. Similarly, we can show that (\ref{product str}) holds for all $x,y\in\huaE$. Thus, $E$ is a product structure on $\huaE$.
\end{proof}

\begin{pro}\label{strict product com pro}
Let $(\huaE,[\cdot,\cdot]_\huaE)$ be a Leibniz algebra. Then
\begin{enumerate}[{\rm (i)}]
  \item An almost product structure $E$ on $\huaE$ satisfying
   \begin{eqnarray}\label{product condition 1}
    E[x,y]_\huaE=[Ex,y]_\huaE,\,\,\,\forall x,y\in\huaE,
   \end{eqnarray}
   if and only if $\huaE$ admits a decomposition: $\huaE=\huaE_+\oplus\huaE_-,$ where $\huaE_+$ and $\huaE_-$ are subalgebras of $\huaE$ such that $[\huaE_+,\huaE_-]_{\huaE}\subset\huaE_+$ and $[\huaE_{-},\huaE_+]_{\huaE}\subset\huaE_{-}$.
  \item An almost product structure $E$ on $\huaE$ satisfying
   \begin{eqnarray}\label{product condition 2}
    E[x,y]_\huaE=[x,Ey]_\huaE,\,\,\,\forall x,y\in\huaE,
   \end{eqnarray}
   if and only if $\huaE$ admits a decomposition: $\huaE=\huaE_+\oplus\huaE_-,$ where $\huaE_+$ and $\huaE_-$ are subalgebras of $\huaE$ such that $[\huaE_+,\huaE_-]_{\huaE}\subset\huaE_-$ and $[\huaE_{-},\huaE_+]_{\huaE}\subset\huaE_{+}$.
\end{enumerate}
\end{pro}

\begin{proof}
 (i) By \eqref{product condition 1} and $E^2=\Id$, we have
\begin{eqnarray*}
[Ex,y]_{\huaE}+[x,Ey]_{\huaE}-E[Ex,Ey]_{\huaE}&=&[Ex,y]_{\huaE}+[x,Ey]_{\huaE}-[E^2 x,Ey]_{\huaE}\\
&=&[Ex,y]_{\huaE}+[x,Ey]_{\huaE}-[x,Ey]_{\huaE}\\
&=&E[x,y]_\huaE,
\end{eqnarray*}
which implies that $E$ is a product structure on $(\huaE,[\cdot,\cdot]_\huaE)$. Therefore, $\huaE$ admits a decomposition: $\huaE=\huaE_+\oplus\huaE_-,$ where $\huaE_+$ and $\huaE_-$ are subalgebras of $\huaE$ by Theorem \ref{product to decomposition}. For all $x\in\huaE_+$ and $\xi\in\huaE_-$, we have $E[x,\xi]_{\huaE}=[Ex,\xi]_{\huaE}=[x,\xi]_{\huaE}$, which implies that $[\huaE_+,\huaE_-]_{\huaE}\subset\huaE_+$. In the same way, we obtain $[\huaE_{-},\huaE_+]_{\huaE}\subset\huaE_{-}$.

Conversely, by Theorem \ref{product to decomposition}, we obtain a product structure $E$ on $\huaE$ by (\ref{decomposition to product}). Since $\huaE_+$ and $\huaE_-$ are subalgebras, $[\huaE_+,\huaE_-]_{\huaE}\subset\huaE_+$ and $[\huaE_{-},\huaE_+]_{\huaE}\subset\huaE_{-}$, we have
\begin{eqnarray*}
E[x+\xi,y+\eta]_\huaE&=&E([x,y]_\huaE+[x,\eta]_\huaE+[\xi,y]_\huaE+[\xi,\eta]_\huaE)=[x,y]_\huaE+[x,\eta]_\huaE-[\xi,y]_\huaE-[\xi,\eta]_\huaE\\
&=&[x-\xi,y+\eta]_\huaE=[E(x+\xi),y+\eta]_\huaE,
\end{eqnarray*}
for all $x,y\in\huaE_+$ and $\xi,\eta\in\huaE_-$.

 (ii) The proof is similar to (i), which is left to readers.
\end{proof}

\begin{defi}
An almost product structure on a Leibniz algebra $(\huaE,[\cdot,\cdot]_\huaE)$ is called a {\bf strict product structure} if  \eqref{product condition 1} and \eqref{product condition 2} hold.
\end{defi}

\begin{pro}
Let $(\huaE,[\cdot,\cdot]_\huaE)$ be a Leibniz algebra. Then there exists a strict product structure on $(\huaE,[\cdot,\cdot]_\huaE)$ if and only if $\huaE$ admits a decomposition:
$$\huaE=\huaE_+\oplus\huaE_-,$$
where $\huaE_+$ and $\huaE_-$ are subalgebras of $\huaE$ such that $[\huaE_+,\huaE_-]_{\huaE}=[\huaE_-,\huaE_+]_{\huaE}=0$, that is, $\huaE$ is a Leibniz algebra direct sum of $\huaE_+$ and $\huaE_-$.
\end{pro}

\begin{proof}
Let $E$ be a strict product structure on $\huaE$. By Theorem \ref{product to decomposition}, we have $\huaE=\huaE_+\oplus\huaE_-$, where $\huaE_+$ and $\huaE_-$ are subalgebras of $\huaE$. Moreover, we have $[\huaE_+,\huaE_-]_{\huaE}=[\huaE_-,\huaE_+]_{\huaE}=0$ by Proposition \ref{strict product com pro}.

Conversely, we obtain a product structure $E$ on $\huaE$ by (\ref{decomposition to product}). Since $\huaE_+$ is a subalgebra, we have
\begin{eqnarray*}
E[x_1,x_2]_\huaE=[x_1,x_2]_\huaE=[Ex_1,x_2]_\huaE=[x_1,Ex_2]_\huaE,\,\,\,\forall x_1,x_2\in\huaE_+,
\end{eqnarray*}
which implies that $E$ satisfies conditions \eqref{product condition 1} and \eqref{product condition 2} on $\huaE_+$. Similarly, $E$ satisfies  \eqref{product condition 1} and \eqref{product condition 2} on $\huaE_-$. Since $[\huaE_+,\huaE_-]_{\huaE}=[\huaE_-,\huaE_+]_{\huaE}=0$, $E$ satisfies  \eqref{product condition 1} and \eqref{product condition 2} on $\huaE$, that is, $E$ is a strict product structure on $\huaE$.
\end{proof}

\begin{pro}
Let $E$ be an almost product structure on the Leibniz algebra $(\huaE,[\cdot,\cdot]_\huaE)$. If $E$ satisfies
\begin{eqnarray}\label{product condition 3}
[x,y]_\huaE=-[Ex,Ey]_\huaE,\,\,\,\forall x,y\in\huaE,
\end{eqnarray}
then $E$ is a product structure on $(\huaE,[\cdot,\cdot]_\huaE)$.
\end{pro}

\begin{proof}
By \eqref{product condition 3} and $E^2=\Id$, we have
\begin{eqnarray*}
[Ex,y]_{\huaE}+[x,Ey]_{\huaE}-E[Ex,Ey]_{\huaE}&=&-[E^2 x,Ey]_{\huaE}+[x,Ey]_{\huaE}-E[Ex,Ey]_{\huaE}\\
&=&-[x,Ey]_{\huaE}+[x,Ey]_{\huaE}-E[Ex,Ey]_{\huaE}\\
&=&E[x,y]_\huaE.
\end{eqnarray*}
Therefore, $E$ is a product structure on $(\huaE,[\cdot,\cdot]_\huaE)$.
\end{proof}

\begin{defi}
An almost product structure  on a Leibniz algebra $(\huaE,[\cdot,\cdot]_\huaE)$ is called an {\bf abelian product structure} if   \eqref{product condition 3} holds.
\end{defi}

\begin{pro}
Let $(\huaE,[\cdot,\cdot]_\huaE)$ be a Leibniz algebra. Then there exists an abelian product structure on $(\huaE,[\cdot,\cdot]_\huaE)$ if and only if $\huaE$ admits a decomposition:
$$\huaE=\huaE_+\oplus\huaE_-,$$
where $\huaE_+$ and $\huaE_-$ are abelian subalgebras of $\huaE$.
\end{pro}

\begin{proof}
Let $E$ be an abelian product structure on $\huaE$. We only need to show that $\huaE_+$ and $\huaE_-$ are abelian subalgebras. By \eqref{product condition 3}, we have
\begin{eqnarray*}
&&[x_1,x_2]_\huaE=-[Ex_1,Ex_2]_{\huaE}=-[x_1,x_2]_{\huaE},\,\,\,\forall x_1,x_2\in\huaE_+,\\
&&[\xi_1,\xi_2]_\huaE=-[E\xi_1,E\xi_2]_{\huaE}=-[-\xi_1,-\xi_2]_{\huaE}=-[\xi_1,\xi_2]_{\huaE},\,\,\,\forall \xi_1,\xi_2\in\huaE_-,
\end{eqnarray*}
which implies that $\huaE_+$ and $\huaE_-$ are abelian subalgebras of $\huaE$.

Conversely, we obtain a product structure defined by \eqref{decomposition to product} according to Theorem \ref{product to decomposition}. Since $\huaE_+$ and $\huaE_-$ are abelian subalgebras of $\huaE$, we have
$[x_1+\xi_1,x_2+\xi_2]_\huaE=-[E(x_1+\xi_1),E(x_2+\xi_2)]_\huaE$ for all $x_1+\xi_1,x_2+\xi_2\in\huaE=\huaE_+\oplus\huaE_-$, which means that $E$ is an abelian product structure on $\huaE$.
\end{proof}

\begin{defi}\label{def of paracomplex str}
A {\bf paracomplex structure} on the Leibniz algebra $(\huaE,[\cdot,\cdot]_{\huaE})$ is a product structure $E$ on $\huaE$ such that the eigenspaces of $\huaE$ associated to the eigenvalues $\pm 1$ have the same dimension, i.e. {\rm dim(}$\huaE_+${\rm)}={\rm dim(}$\huaE_-${\rm)}.
\end{defi}

\begin{pro}\label{paracomplex pro}
Let $(\huaL,\lhd,\rhd)$ be a Leibniz-dendriform algebra. Then there is a paracomplex structure $E$ on the phase space $\huaL\ltimes_{L^*_\lhd,-L^*_\lhd-R^*_\rhd}\huaL^*$ given by
\begin{eqnarray}\label{paracomplex str}
E(x+\xi)=x-\xi,\,\,\,\forall x\in\huaL,\xi\in\huaL^*.
\end{eqnarray}
\end{pro}

\begin{proof}
It is obvious that $E^2=\Id$. Moreover, we have $(\huaL\ltimes_{L^*_\lhd,-L^*_\lhd-R^*_\rhd}\huaL^*)_+=\huaL, (\huaL\ltimes_{L^*_\lhd,-L^*_\lhd-R^*_\rhd}\huaL^*)_-=\huaL^*$ and they are two subalgebras of the semidirect product Leibniz algebra $\huaL\ltimes_{L^*_\lhd,-L^*_\lhd-R^*_\rhd}\huaL^*$. By Theorem \ref{product to decomposition}, $E$ is a product structure on $\huaL\ltimes_{L^*_\lhd,-L^*_\lhd-R^*_\rhd}\huaL^*$. Since dim($\huaL$)=dim($\huaL^*$), $E$ is a paracomplex structure on $\huaL\ltimes_{L^*_\lhd,-L^*_\lhd-R^*_\rhd}\huaL^*$.
\end{proof}

\begin{ex}\label{pro example 1}
Consider the $4$-dimensional Leibniz algebra $({\huaE},[\cdot,\cdot]_{\huaE})$ given in {\rm Example \ref{sym example 1}}. Then
\begin{center}
$E_1=$
$\begin{pmatrix}
1 & 0 & 0 & 0\\
0 & 1 & 0 & 0\\
0 & 0 & -1 & 0\\
0 & 0 & 0 & -1
\end{pmatrix}$,\,\,\,
$E_2=$
$\begin{pmatrix}
-1 & 0 & 0 & 0\\
0 & -1 & 0 & 0\\
0 & 0 & 1 & 0\\
0 & 0 & 0 & 1
\end{pmatrix}$,\,\,\,
$E_3=$
$\begin{pmatrix}
1 & 0 & 0 & 0\\
0 & -1 & 0 & 0\\
0 & 0 & -1 & 0\\
0 & 0 & 0 & 1
\end{pmatrix}$,\\
$E_4=$
$\begin{pmatrix}
1 & 0 & 0 & 0\\
0 & -1 & 0 & 0\\
0 & 0 & -1 & 0\\
0 & 0 & 0 & -1
\end{pmatrix}$,\,\,\,
$E_5=$
$\begin{pmatrix}
1 & 0 & 0 & 0\\
0 & -1 & 0 & 0\\
0 & 0 & 1 & 0\\
0 & 0 & 0 & 1
\end{pmatrix}$,\,\,\,
$E_6=$
$\begin{pmatrix}
-1 & 0 & 0 & 0\\
0 & 1 & 0 & 0\\
0 & 0 & -1 & 0\\
0 & 0 & 0 & -1
\end{pmatrix}$
\end{center}
are product structures on ${\huaE}$, where $E_1,E_2,E_3,E_4$ are abelian, $E_5,E_6$ are strict and $E_1,E_2,E_3$ are paracomplex structures.
\end{ex}

\begin{ex}
Let $T:\huaV\rightarrow\huaE$ be a relative Rota-Baxter operator on $\huaE$ with respect to the representation $(\huaV;l,r)$. By Proposition \ref{Leibniz-dendriform}, $(\huaV,[\cdot,\cdot]_{\lhd, \rhd})$ is the sub-adjacent Leibniz algebra of Leibniz-dendriform algebra $(\huaV,\lhd,\rhd)$. Then we obtain a Leibniz algebra $(\huaE\oplus\huaV,[\cdot,\cdot]_{\bowtie})$ which is studied in \cite{Tang-Sheng,Tang-Sheng2}, where $[\cdot,\cdot]_{\bowtie}$ is given by
\begin{eqnarray*}
[x+u,y+v]_{\bowtie}=
[x,y]_\huaE+([Tu,y]_\huaE-T(r_y u))+([x,Tv]_\huaE-T(l_x v))+[u,v]_{\lhd,\rhd}+l_x v+r_y u.
\end{eqnarray*}
Therefore, $E$ given by $E(x+u)=x-u$ for all $x\in\huaE,u\in\huaV$ is a product structure on $\huaE\oplus\huaV$.
\end{ex}

\section{Complex structures on Leibniz algebras}\label{sec:com}

In this section, we introduce the notion of a complex structure on a Leibniz algebra. Parallel to the case of product structures, we also find two special complex structures.

\begin{defi}\label{def of complex}
Let $(\huaE,[\cdot,\cdot]_{\huaE})$ be a real Leibniz algebra. An {\bf almost complex structure} on $\huaE$ is a linear endomorphism $J:\huaE\rightarrow\huaE$ satisfying $J^2=-\Id$. A {\bf complex structure} on $\huaE$ is an almost complex structure $J$ satisfying the following integrability condition:
\begin{eqnarray}\label{complex str}
J[x,y]_{\huaE}=[Jx,y]_{\huaE}+[x,Jy]_{\huaE}+J[Jx,Jy]_{\huaE},\,\,\,\forall x,y\in\huaE.
\end{eqnarray}
\end{defi}

\begin{rmk}
One can understand a complex structure on a Leibniz algebra as a Nijenhuis operator $J$ on the Leibniz algebra satisfying $J^2=-\Id$.
\end{rmk}

\begin{rmk}
One can use {\rm Definition \ref{def of complex}} to define the notion of a complex structure on a complex Leibniz algebra, considering $J$ to be $\mathds{C}$-linear. We do not pay more attention to such case since {\rm Proposition \ref{relation between product and complex str}} will show that there is a one-to-one correspondence between $\mathds{C}$-linear complex structures and product structures.
\end{rmk}

Consider $\huaE_{\mathds{C}}=\huaE\otimes_{\mathds{R}}\mathds{C}\cong\{x+iy\mid x,y\in\huaE\}$, the complexification of the real Leibniz algebra $\huaE$, which turns out to be a complex Leibniz algebra by extending the Leibniz bracket on $\huaE$ complex linearly, and we denote it by $(\huaE_{\mathds{C}},[\cdot,\cdot]_{\huaE_\mathds{C}})$. We denote by $\sigma$ the conjugation in $\huaE_{\mathds{C}}$ with respect to the real form $\huaE$, that is, $\sigma(x+iy)=x-iy$, where $x,y\in\huaE$. Obviously, $\sigma$ is a complex antilinear, involutive automorphism of the complex vector space $\huaE_{\mathds{C}}$. There is an equivalent description of the integrability condition given in Definition \ref{def of complex}.

\begin{thm}\label{complex to decomposition}
Let $(\huaE,[\cdot,\cdot]_{\huaE})$ be a real Leibniz algebra. Then $\huaE$ has a complex structure if and only if $\huaE_{\mathds{C}}$ admits a decomposition:
\begin{eqnarray}
\huaE_{\mathds{C}}=\mathfrak{q}\oplus \mathfrak{p},
\end{eqnarray}
where $\mathfrak{q}$ and $\mathfrak{p}=\sigma(\mathfrak{q})$ are complex subalgebras of $\huaE_{\mathds{C}}$.
\end{thm}

\begin{proof}
Let $J$ be a complex structure on the real Leibniz algebra $(\huaE,[\cdot,\cdot]_{\huaE})$. We extend $J$ complex linearly, which is denoted by $J_{\mathds{C}}$, that is, $J_{\mathds{C}}:\huaE_{\mathds{C}}\rightarrow\huaE_{\mathds{C}}$ is defined by
\begin{eqnarray}\label{JC_1}
J_{\mathds{C}}(x+iy)=Jx+iJy,\,\,\,\forall x,y\in\huaE.
\end{eqnarray}
It is easy to verify that $J_{\mathds{C}}$ is a complex linear endomorphism on $\huaE_{\mathds{C}}$ satisfying $J_{\mathds{C}}^2=-\Id$ and the integrability condition (\ref{complex str}) on $\huaE_{\mathds{C}}$. Denoted by $\huaE_{\pm i}$ the corresponding eigenspaces of $\huaE_{\mathds{C}}$ associated to the eigenvalues $\pm i$ and there holds:
$$\huaE_{\mathds{C}}=\huaE_i\oplus\huaE_{-i}.$$
It is straightforward to see that $\huaE_i=\{x-iJx\mid x\in\huaE\}$ and $\huaE_{-i}=\{x+iJx\mid x\in\huaE\}$.
Therefore, we have $\huaE_{-i}=\sigma(\huaE_{i})$.
For all $X,Y\in\huaE_{i}$, we have
\begin{eqnarray*}
J_{\mathds{C}}[X,Y]_{\huaE_\mathds{C}}&=&[J_{\mathds{C}}X,Y]_{\huaE_{\mathds{C}}}+[X,J_{\mathds{C}}Y]_{\huaE_{\mathds{C}}}+J_{\mathds{C}}[J_{\mathds{C}}X,J_{\mathds{C}}Y]_{\huaE_{\mathds{C}}}\\
&=&[iX,Y]_{\huaE_{\mathds{C}}}+[X,iY]_{\huaE_{\mathds{C}}}+J_{\mathds{C}}[iX,iY]_{\huaE_{\mathds{C}}}\\
&=&2i[X,Y]_{\huaE_{\mathds{C}}}-J_{\mathds{C}}[X,Y]_{\huaE_{\mathds{C}}},
\end{eqnarray*}
which means that $[X,Y]_{\huaE_\mathds{C}}\in\huaE_{i}$. Thus $\huaE_{i}$ is a subalgebra of $\huaE_\mathds{C}$. Similarly, we can show that $\huaE_{-i}$ is also a subalgebra.

Conversely, define a complex linear endomorphism $J_{\mathds{C}}:\huaE_{\mathds{C}}\rightarrow\huaE_{\mathds{C}}$ by
\begin{eqnarray}\label{complex JC}
J_{\mathds{C}}(X+\sigma(Y))&=&iX-i\sigma(Y),\,\,\,\forall X,Y\in\mathfrak{q}.
\end{eqnarray}
Since $\sigma$ is a complex antilinear, involutive automorphism of $\huaE_{\mathds{C}}$, we have
\begin{eqnarray*}
J_{\mathds{C}}^2(X+\sigma(Y))=J_{\mathds{C}}(iX-i\sigma(Y))=J_{\mathds{C}}(iX+\sigma(iY))
=i(iX)-i\sigma(iY)=-X-\sigma(Y),
\end{eqnarray*}
i.e. $J_{\mathds{C}}^2=-\Id$. Since $\mathfrak{q}$ is a subalgebra of $\huaE_{\mathds{C}}$, for all $X,Y\in\mathfrak{q}$, we have
\begin{eqnarray*}
&&[J_{\mathds{C}}X,Y]_{\huaE_{\mathds{C}}}+[X,J_{\mathds{C}}Y]_{\huaE_{\mathds{C}}}+J_{\mathds{C}}[J_{\mathds{C}}X,J_{\mathds{C}}Y]_{\huaE_{\mathds{C}}}\\
&=&[iX,Y]_{\huaE_{\mathds{C}}}+[X,iY]_{\huaE_{\mathds{C}}}+J_{\mathds{C}}[iX,iY]_{\huaE_{\mathds{C}}}\\
&=&2i[X,Y]_{\huaE_{\mathds{C}}}-J_{\mathds{C}}[X,Y]_{\huaE_{\mathds{C}}}=i[X,Y]_{\huaE_{\mathds{C}}}\\
&=&J_{\mathds{C}}[X,Y]_{\huaE_{\mathds{C}}},
\end{eqnarray*}
which implies that $J_{\mathds{C}}$ satisfies (\ref{complex str}) for all $X,Y\in\mathfrak{q}$. Similarly, we can show that $J_{\mathds{C}}$ satisfies (\ref{complex str}) for all $\mathcal{X},\mathcal{Y}\in\huaE_{\mathds{C}}$. Since $\huaE_{\mathds{C}}=\mathfrak{q}\oplus \mathfrak{p}$, we can write $\mathcal{X}\in\huaE_{\mathds{C}}$ as $\mathcal{X}=X+\sigma(Y)$ for some $X,Y\in\mathfrak{q}$. Since $\sigma$ is a complex antilinear, involutive automorphism of $\huaE_{\mathds{C}}$, we have
\begin{eqnarray*}
(J_{\mathds{C}}\circ\sigma)(X+\sigma(Y))&=&J_{\mathds{C}}(Y+\sigma(X))=iY-i\sigma(X)=\sigma(iX-i\sigma(Y))\\
&=&(\sigma\circ J_{\mathds{C}})(X+\sigma(Y)),
\end{eqnarray*}
which means that $J_{\mathds{C}}\circ\sigma=\sigma\circ J_{\mathds{C}}$. Since $\sigma(\mathcal{X})=\mathcal{X}$ if and only if $\mathcal{X}\in\huaE$, i.e. the set of fixed points of $\sigma$ is the real vector space $\huaE$, then for all $x\in\huaE$, $\sigma(J_{\mathds{C}}(x))=J_{\mathds{C}}(\sigma(x))=J_{\mathds{C}}(x)$ implies that $J_{\mathds{C}}(x)\in\huaE$, which makes $J\triangleq J_{\mathds{C}}{\mid_{\huaE}}\in\gl(\huaE)$ well-defined. It is obvious that $J$ is a complex structure on $\huaE$ since $J_{\mathds{C}}$ satisfies (\ref{complex str}) and $J_{\mathds{C}}^2=-\Id$.
\end{proof}

\begin{pro}\label{strict complex com pro}
Let $J$ be an almost complex structure on the real Leibniz algebra $(\huaE,[\cdot,\cdot]_\huaE)$ and $\huaE_i$, $\huaE_{-i}$ the eigenspaces of $\huaE$ associated to the eigenvalues $\pm i$.
\begin{enumerate}[{\rm (i)}]
  \item If $J$ satisfies
   \begin{eqnarray}\label{complex condition 1}
    J[x,y]_\huaE=[Jx,y]_\huaE,\,\,\,\forall x,y\in\huaE,
   \end{eqnarray}
   then $J$ is a complex structure on $(\huaE,[\cdot,\cdot]_\huaE)$, such that $[\huaE_i,\huaE_{-i}]_{\huaE_\mathds{C}}\subset\huaE_i$ and $[\huaE_{-i},\huaE_i]_{\huaE_\mathds{C}}\subset\huaE_{-i}$.
  \item If $J$ satisfies
   \begin{eqnarray}\label{complex condition 2}
    J[x,y]_\huaE=[x,Jy]_\huaE,\,\,\,\forall x,y\in\huaE,
   \end{eqnarray}
   then $J$ is a complex structure on $(\huaE,[\cdot,\cdot]_\huaE)$, such that $[\huaE_i,\huaE_{-i}]_{\huaE_\mathds{C}}\subset\huaE_{-i}$ and $[\huaE_{-i},\huaE_i]_{\huaE_\mathds{C}}\subset\huaE_i$.
\end{enumerate}
\end{pro}

\begin{proof}
 (i)  By \eqref{complex condition 1} and $J^2=-\Id$, we have
\begin{eqnarray*}
[Jx,y]_{\huaE}+[x,Jy]_{\huaE}+J[Jx,Jy]_{\huaE}&=&[Jx,y]_{\huaE}+[x,Jy]_{\huaE}+[J^2 x,Jy]_{\huaE}\\
&=&[Jx,y]_{\huaE}+[x,Jy]_{\huaE}-[x,Jy]_{\huaE}\\
&=&J[x,y]_\huaE.
\end{eqnarray*}
Therefore, $J$ is a complex structure on $(\huaE,[\cdot,\cdot]_\huaE)$. For all $x\in\huaE_i$ and $Y\in\huaE_{-i}$, we have $J_{\mathds{C}}[x,Y]_{\huaE_\mathds{C}}=[J_{\mathds{C}}(x),Y]_{\huaE_\mathds{C}}=[ix,Y]_{\huaE_\mathds{C}}=i[x,Y]_{\huaE_\mathds{C}}$  , which implies that $[\huaE_i,\huaE_{-i}]_{\huaE_\mathds{C}}\subset\huaE_i$. In the same way, we obtain $[\huaE_{-i},\huaE_i]_{\huaE_\mathds{C}}\subset\huaE_{-i}$.

  (ii) The proof is similar to (i), which is left to readers.
\end{proof}

\begin{defi}
An almost complex structure on the real Leibniz algebra $(\huaE,[\cdot,\cdot]_\huaE)$ is called a ${\bf strict~complex~structure}$ if  \eqref{complex condition 1} and \eqref{complex condition 2} hold.
\end{defi}

\begin{pro}
Let $(\huaE,[\cdot,\cdot]_\huaE)$ be a real Leibniz algebra. Then there exists a strict complex structure on $(\huaE,[\cdot,\cdot]_\huaE)$ if and only if $\huaE_{\mathds{C}}$ admits a decomposition:
$$\huaE_{\mathds{C}}=\mathfrak{q}\oplus \mathfrak{p},$$
where $\mathfrak{q}$ and $\mathfrak{p}=\sigma(\mathfrak{q})$ are complex subalgebras of $\huaE_{\mathds{C}}$ such that $[\mathfrak{q},\mathfrak{p}]_{\huaE_\mathds{C}}=[\mathfrak{p},\mathfrak{q}]_{\huaE_\mathds{C}}=0$, that is, $\huaE_{\mathds{C}}$ is a Leibniz algebra direct sum of $\mathfrak{q}$ and $\mathfrak{p}$.
\end{pro}

\begin{proof}
Let $J$ be a strict complex structure on the real Leibniz algebra $(\huaE,[\cdot,\cdot]_\huaE)$. By Theorem \ref{complex to decomposition} and Proposition \ref{strict complex com pro}, we have $\huaE_{\mathds{C}}=\huaE_i\oplus\huaE_{-i}$, where $[\huaE_i,\huaE_{-i}]_{\huaE_\mathds{C}}=[\huaE_{-i},\huaE_i]_{\huaE_\mathds{C}}=0$.

Conversely, we define a complex linear endomorphism $J_{\mathds{C}}:\huaE_{\mathds{C}}\rightarrow\huaE_{\mathds{C}}$ by (\ref{complex JC}), where $J_{\mathds{C}}^2=-\Id$ by direct computation. Since $\mathfrak{q}$ is a complex subalgebra of $\huaE_{\mathds{C}}$, then
$$J_{\mathds{C}}[x,y]_{\huaE_\mathds{C}}=i[x,y]_{\huaE_\mathds{C}}
=[J_{\mathds{C}}(x),y]_{\huaE_\mathds{C}}=[x,J_{\mathds{C}}(y)]_{\huaE_\mathds{C}},\,\,\,\forall x,y\in\mathfrak{q},$$
which implies that $J_{\mathds{C}}$ satisfies   (\ref{complex condition 1}) and (\ref{complex condition 2}) on $\mathfrak{q}$. Similarly, $J_{\mathds{C}}$ also satisfies  (\ref{complex condition 1}) and (\ref{complex condition 2}) on $\mathfrak{p}$. Since $\huaE_{\mathds{C}}$ is a Leibniz algebra direct sum of $\mathfrak{q}$ and $\mathfrak{p}$,  $J_{\mathds{C}}$ satisfies   (\ref{complex condition 1}) and (\ref{complex condition 2}) on $\huaE_{\mathds{C}}$. By the proof of Theorem \ref{complex to decomposition}, we obtain that $J\triangleq J_{\mathds{C}}{\mid_{\huaE}}$ is a strict complex structure on the real Leibniz algebra $(\huaE,[\cdot,\cdot]_\huaE)$.
\end{proof}

Let $J$ be an almost complex structure on the real Leibniz algebra $(\huaE,[\cdot,\cdot]_\huaE)$. We can define a complex vector space structure on the real vector space $\huaE$ by
\begin{eqnarray}\label{real to complex}
(a+bi)x\triangleq ax+bJx,\,\,\,\forall a,b\in\mathds{R},\,x\in\huaE.
\end{eqnarray}
Define two maps $\varphi:\huaE\rightarrow\huaE_i$ and $\psi:\huaE\rightarrow\huaE_{-i}$ as follows:
\begin{eqnarray*}
\varphi(x)&=&\frac{1}{2}(x-iJx),\\
\psi(x)&=&\frac{1}{2}(x+iJx).
\end{eqnarray*}
It is straightforward to deduce that $\varphi$ is complex linear isomorphism and $\psi=\sigma\circ\varphi$ is a complex antilinear isomorphism between complex vector spaces.

Let $J$ be a strict complex structure on the real Leibniz algebra $(\huaE,[\cdot,\cdot]_\huaE)$. By \eqref{complex condition 1}, \eqref{complex condition 2} and \eqref{real to complex}, we have
\begin{eqnarray*}
&&[(a+bi)x,y]_\huaE=[ax+bJx,y]_\huaE=a[x,y]_\huaE+b[Jx,y]_\huaE=a[x,y]_\huaE+bJ[x,y]_\huaE=(a+bi)[x,y]_\huaE,\\
&&[x,(a+bi)y]_\huaE=[x,ay+bJy]_\huaE=a[x,y]_\huaE+b[x,Jy]_\huaE=a[x,y]_\huaE+bJ[x,y]_\huaE=(a+bi)[x,y]_\huaE,
\end{eqnarray*}
which implies that $\huaE$ is a complex Leibniz algebra with the complex vector space structure defined above.

Conversely, let $(\huaE,[\cdot,\cdot]_\huaE)$ be a complex Leibniz algebra and let $\huaE_{\mathds{R}}$ be the underlying real Leibniz algebra. Then there is an endomorphism $J$ of $\huaE_{\mathds{R}}$ given by multiplication by $i$, that is, $Jx=ix$ for all $x\in\huaE_{\mathds{R}}$. It is straightforward to see that $J$ satisfies  \eqref{complex condition 1} and \eqref{complex condition 2} and $J^2=-\Id$, i.e. $J$ is a strict complex structure on $\huaE_{\mathds{R}}$, which means that the real Leibniz algebra underlying a complex structure naturally endows with a strict complex structure.

Let $J$ be a complex structure on $\huaE$. Define a new bracket $[\cdot,\cdot]_J:\huaE\otimes\huaE\rightarrow\huaE$ by
\begin{eqnarray}\label{new bracket}
[x,y]_J=\frac{1}{2}([x,y]_\huaE-[Jx,Jy]_\huaE),\,\,\,\forall x,y\in\huaE.
\end{eqnarray}

\begin{pro}\label{complex to strict complex}
Let $J$ be a complex structure on the real Leibniz algebra $(\huaE,[\cdot,\cdot]_\huaE)$. Then $(\huaE,[\cdot,\cdot]_J)$ is a real Leibniz algebra. Moreover, $J$ is a strict complex structure on  $(\huaE,[\cdot,\cdot]_J)$ and the corresponding complex Leibniz algebra is isomorphic to the complex Leibniz algebra $\huaE_i$.
\end{pro}

\begin{proof}
By \eqref{complex str}, for all $x,y\in\huaE$, we have
\begin{eqnarray*}
[\varphi(x),\varphi(y)]_{\huaE_{\mathds{C}}}&=&\frac{1}{4}[x-iJx,y-iJy]_{\huaE_{\mathds{C}}}\\
&=&\frac{1}{4}([x,y]_\huaE-[Jx,Jy]_\huaE)-\frac{1}{4}i([Jx,y]_\huaE+[x,Jy]_\huaE)\\
&=&\frac{1}{4}([x,y]_\huaE-[Jx,Jy]_\huaE)-\frac{1}{4}iJ([x,y]_\huaE-[Jx,Jy]_\huaE)\\
&=&\frac{1}{2}([x,y]_J)-\frac{1}{2}iJ([x,y]_J)\\
&=&\varphi([x,y]_J).
\end{eqnarray*}
Since $J$ is a complex structure, $\huaE_i$ is a Leibniz subalgebra by Theorem \ref{complex to decomposition}. Thus $[x,y]_J=\varphi^{-1}[\varphi(x),\varphi(y)]_{\huaE_{\mathds{C}}}$ implies that $(\huaE,[\cdot,\cdot]_J)$ is a real Leibniz algebra. Furthermore, since $\varphi$ is a complex Leibniz algebra isomorphism, it follows that the complex Leibniz algebra $(\huaE,[\cdot,\cdot]_J)$ is isomorphic to the complex Leibniz algebra $\huaE_i$.

By (\ref{complex str}) and (\ref{new bracket}), for all $x,y\in\huaE$, we have
\begin{eqnarray*}
J[x,y]_J&=&\frac{1}{2}J([x,y]_\huaE-[Jx,Jy]_\huaE)=\frac{1}{2}([Jx,y]_\huaE+[x,Jy]_\huaE)\\
&=&[Jx,y]_J\\
&=&[x,Jy]_J,
\end{eqnarray*}
which implies that $J$ is a strict complex structure on  $(\huaE,[\cdot,\cdot]_J)$.
\end{proof}

\begin{pro}\label{strict complex equvialent}
Let $J$ be a complex structure on the real Leibniz algebra $(\huaE,[\cdot,\cdot]_\huaE)$. Then $J$ is a strict complex structure on $(\huaE,[\cdot,\cdot]_\huaE)$ if and only if $[\cdot,\cdot]_\huaE=[\cdot,\cdot]_J$.
\end{pro}

\begin{proof}
Let $J$ be a strict complex structure on $(\huaE,[\cdot,\cdot]_\huaE)$, then by \eqref{complex condition 1} and \eqref{complex condition 2}, we have
$$[x,y]_J=\frac{1}{2}([x,y]_\huaE-[Jx,Jy]_\huaE)
=\frac{1}{2}([x,y]_\huaE-J[x,Jy]_\huaE)=\frac{1}{2}([x,y]_\huaE-J^2[x,y]_\huaE)
=[x,y]_\huaE.$$

Conversely, if $[\cdot,\cdot]_\huaE=[\cdot,\cdot]_J$, then by Proposition \ref{complex to strict complex} we have
\begin{eqnarray*}
J[x,y]_\huaE&=&J[x,y]_J=[Jx,y]_J=[Jx,y]_\huaE,\\
J[x,y]_\huaE&=&J[x,y]_J=[x,Jy]_J=[x,Jy]_\huaE,
\end{eqnarray*}
which implies that $J$ is a strict complex structure on $(\huaE,[\cdot,\cdot]_\huaE)$.
\end{proof}

\begin{pro}
Let $J$ be an almost complex structure on the real Leibniz algebra $(\huaE,[\cdot,\cdot]_\huaE)$. If $J$ satisfies
\begin{eqnarray}\label{complex condition 3}
[x,y]_\huaE=[Jx,Jy]_\huaE,\,\,\,\forall x,y\in\huaE,
\end{eqnarray}
then $J$ is a complex structure on $(\huaE,[\cdot,\cdot]_\huaE)$.
\end{pro}

\begin{proof}
By \eqref{complex condition 3} and $J^2=-\Id$, we have
\begin{eqnarray*}
[Jx,y]_{\huaE}+[x,Jy]_{\huaE}+J[Jx,Jy]_{\huaE}&=&[J^2 x,Jy]_{\huaE}+[x,Jy]_{\huaE}+J[Jx,Jy]_{\huaE}\\
&=&-[x,Jy]_{\huaE}+[x,Jy]_{\huaE}+J[Jx,Jy]_{\huaE}\\
&=&J[x,y]_\huaE,
\end{eqnarray*}
which implies that $J$ is a complex structure on $(\huaE,[\cdot,\cdot]_\huaE)$.
\end{proof}

\begin{defi}
An almost complex structure on the real Leibniz algebra $(\huaE,[\cdot,\cdot]_\huaE)$ is called an ${\bf abelian~complex~structure}$ if \eqref{complex condition 3} holds.
\end{defi}

\begin{cor}\label{complex to abelian str}
Let $(\huaE,[\cdot,\cdot]_\huaE)$ be a real Leibniz algebra with an abelian complex structure $J$. Then $(\huaE,[\cdot,\cdot]_J)$ is a real abelian  Leibniz algebra.
\end{cor}

\begin{pro}
Let $(\huaE,[\cdot,\cdot]_\huaE)$ be a real Leibniz algebra. Then there exists an abelian complex structure on $(\huaE,[\cdot,\cdot]_\huaE)$ if and only if $\huaE_{\mathds{C}}$ admits a decomposition:
$$\huaE_{\mathds{C}}=\mathfrak{q}\oplus \mathfrak{p},$$
where $\mathfrak{q}$ and $\mathfrak{p}=\sigma(\mathfrak{q})$ are complex abelian subalgebras of $\huaE_{\mathds{C}}$.
\end{pro}

\begin{proof}
Let $J$ be an abelian complex structure on $(\huaE,[\cdot,\cdot]_\huaE)$. By Proposition \ref{complex to decomposition}, we have $\huaE_{\mathds{C}}=\huaE_i\oplus\huaE_{-i}$. Since $\varphi:(\huaE,[\cdot,\cdot]_J)\rightarrow(\huaE_i,[\cdot,\cdot]_{\huaE_{\mathds{C}}})$ is a complex Leibniz algebra isomorphism by Proposition \ref{complex to strict complex}, then Corollary \ref{complex to abelian str} implies that $\huaE_i$ is an abelian subalgebra of $\huaE_{\mathds{C}}$. Similarly, we can also show that $\psi:(\huaE,[\cdot,\cdot]_J)\rightarrow(\huaE_{-i},[\cdot,\cdot]_{\huaE_{\mathds{C}}})$ is a complex Leibniz algebra isomorphism by the proof of Proposition \ref{complex to strict complex}, which implies that $\huaE_{-i}$ is an abelian subalgebra.

Conversely, there is a complex structure $J$ on $\huaE$ by Proposition \ref{complex to decomposition}. Then we obtain a complex Leibniz algebra isomorphism $\varphi$ from $(\huaE,[\cdot,\cdot]_J)$ to $(\mathfrak{q},[\cdot,\cdot]_{\huaE_{\mathds{C}}})$ by Proposition \ref{complex to strict complex}. Thus, $(\huaE,[\cdot,\cdot]_J)$ is an abelian Leibniz algebra. By the definition of $[\cdot,\cdot]_J$, $J$ is an abelian complex structure on $\huaE$.
\end{proof}

\begin{ex}
Consider the $4$-dimensional real Leibniz algebra $(\huaE,[\cdot,\cdot]_{\huaE})$ given with respect to the basis $\{e_1,e_2,e_3,e_4\}$ by:
\begin{eqnarray*}
[e_1,e_1]_{\huaE}=[e_2,e_2]_{\huaE}=e_3.
\end{eqnarray*}
Then
\begin{center}
$J_1=$
$\begin{pmatrix}
0 & -1 & 0 & 0\\
1 & 0 & 0 & 0\\
0 & 0 & 0 & -1\\
0 & 0 & 1 & 0
\end{pmatrix}$,\,\,\,
$J_2=$
$\begin{pmatrix}
0 & -1 & 0 & 0\\
1 & 0 & 0 & 0\\
0 & 0 & 0 & 1\\
0 & 0 & -1 & 0
\end{pmatrix}$,\,\,\,
$J_3=$
$\begin{pmatrix}
0 & 1 & 0 & 0\\
-1 & 0 & 0 & 0\\
0 & 0 & 0 & -1\\
0 & 0 & 1 & 0
\end{pmatrix}$,\,\,\,
$J_4=$
$\begin{pmatrix}
0 & 1 & 0 & 0\\
-1 & 0 & 0 & 0\\
0 & 0 & 0 & 1\\
0 & 0 & -1 & 0
\end{pmatrix}$
\end{center}
are abelian complex structures on ${\huaE}$.
\end{ex}

\section{Para-K\"{a}hler structures on Leibniz algebras}\label{sec:sympro}

In this section, we introduce the notion of a para-K\"{a}hler structure on a Leibniz algebras by adding a compatibility condition between a symplectic structure and a paracomplex structure on the Leibniz algebra. A para-K\"{a}hler structure gives rise to a pseudo-Riemannian structure. We introduce the notion of a Levi-Civita product associated to a pseudo-Riemannian Leibniz algebra and give its precise formulas using the decomposition of the original Leibniz algebra.

\begin{defi}
Let $\huaB$ be a symplectic structure and $E$ a paracomplex structure on the Leibniz algebra $(\huaE,[\cdot,\cdot]_\huaE)$. The triple $(\huaE,\huaB,E)$ is called a {\bf para-K\"{a}hler Leibniz algebra} if the following equality holds:
\begin{eqnarray}
\huaB(Ex,Ey)=-\huaB(x,y),\,\,\,\forall x,y\in\huaE.\label{para-Kahler str}
\end{eqnarray}
\end{defi}

\begin{ex}
Consider the symplectic structures and the product structures on the $4$-dimensional Leibniz algebra $({\huaE},[\cdot,\cdot]_{\huaE})$ given in {\rm Example \ref{sym example 1}} and {\rm Example \ref{pro example 1}} respectively. By direct computation, we obtain that $\{\huaB,E_i\}$ for $i=1,2$ are para-K\"{a}hler structures on ${\huaE}$ if and only if the nondegenerate and symmetric matrix of $\huaB$ is of the following form:
\begin{center}
$\begin{pmatrix}
0 & 0 & * & *\\
0 & 0 & * & 0\\
* & * & 0 & 0\\
* & 0 & 0 & 0
\end{pmatrix}$.
\end{center}
\end{ex}

\begin{pro}
Let $(\huaE,\huaB,E)$ be a para-K\"{a}hler Leibniz algebra.
\begin{enumerate}[{\rm (i)}]
  \item If $E$ is strict, then the compatible Leibniz-dendriform algebra structure defined by $\huaB$ on $\huaE$ satisfies
      \begin{eqnarray*}
       E(x\lhd y)=x\lhd Ey,\,\,\,E(x\rhd y)=Ex\rhd y,\,\,\,\forall x,y\in\huaE.
      \end{eqnarray*}
  \item If $E$ is abelian, then the compatible Leibniz-dendriform algebra structure defined by $\huaB$ on $\huaE$ satisfies
      \begin{eqnarray*}
       E(x\lhd y)=Ex\lhd y,\,\,\,E(x\rhd y)=x\rhd Ey,\,\,\,\forall x,y\in\huaE.
      \end{eqnarray*}
\end{enumerate}
\end{pro}

\begin{proof}
  (i)  For all $x,y,z\in\huaE$, we have
\begin{eqnarray*}
\huaB(E(x\lhd y),z)&=&-\huaB(x\lhd y,Ez)=\huaB(y,[x,Ez]_\huaE)=\huaB(y,E[x,z]_\huaE)\\
&=&-\huaB(Ey,[x,z]_\huaE)=\huaB(x\lhd Ey,z),\\
\huaB(E(x\rhd y),z)&=&-\huaB(x\rhd y,Ez)=-\huaB(x,[y,Ez]_\huaE+[Ez,y]_\huaE)\\
&=&-\huaB(x,E([y,z]_\huaE+[z,y]_\huaE))\\
&=&\huaB(Ex,[y,z]_\huaE+[z,y]_\huaE)\\
&=&\huaB(Ex\rhd y,z),
\end{eqnarray*}
which implies that
\begin{eqnarray*}
       E(x\lhd y)=x\lhd Ey,\,\,\,E(x\rhd y)=Ex\rhd y,\,\,\,\forall x,y\in\huaE.
      \end{eqnarray*}

 (ii) The proof is similar to (i), which is left to readers.
\end{proof}

\begin{thm}\label{para-Kahler if and only if}
Let $(\huaE,\huaB)$ be a symplectic Leibniz algebra. Then there exists a paracomplex structure $E$ on $\huaE$ such that $(\huaE,\huaB,E)$ is a para-K\"{a}hler Leibniz algebra if and only if there exist two isotropic Leibniz algebras $\huaE_+$ and $\huaE_-$ such that $\huaE=\huaE_+\oplus\huaE_-$ as the direct sum of vector spaces.
\end{thm}

\begin{proof}
Let $(\huaE,\huaB,E)$ be a para-K\"{a}hler Leibniz algebra. Since $E$ is a paracomplex structure on $\huaE$, we have $\huaE=\huaE_+\oplus\huaE_-$, where $\huaE_+$ and $\huaE_-$ are Leibniz subalgebras of $\huaE$. For all $x_1,x_2\in\huaE_+$, we have $\huaB(x_1,x_2)=\huaB(Ex_1,Ex_2)=-\huaB(x_1,x_2)$ by (\ref{para-Kahler str}), which implies that $\huaB(\huaE_+,\huaE_+)=0$. Thus, $\huaE_+$ is isotropic. Similarly, $\huaE_-$ is also isotropic.

Conversely, since $\huaE=\huaE_+\oplus\huaE_-$ as the direct sum of vector spaces, where $\huaE_+$ and $\huaE_-$ are subalgebras, there is a product structure $E$ on $\huaE$ defined by (\ref{decomposition to product}). Since $\huaE_+$ and $\huaE_-$ are isotropic, we obtain that
dim $\huaE_+$ = dim $\huaE_-$. Thus $E$ is a paracomplex structure on $\huaE$. For all $x_1,x_2\in\huaE_+$ and $\xi_1,\xi_2\in\huaE_-$, we have
\begin{eqnarray*}
\huaB(E(x_1+\xi_1),E(x_2+\xi_2))&=&\huaB(x_1-\xi_1,x_2-\xi_2)=-\huaB(x_1,\xi_2)-\huaB(\xi_2,x_1)\\
&=&-\huaB(x_1+\xi_1,x_2+\xi_2).
\end{eqnarray*}
Thus, $(\huaE,\huaB,E)$ is a para-K\"{a}hler Leibniz algebra.
\end{proof}

\begin{pro}\label{phase space to para-Kahler}
Let $(\huaE,[\cdot,\cdot]_\huaE)$ be a Leibniz algebra and $(\huaE\oplus\huaE^*,\huaB)$ its phase space, where $\huaB$ is given by \eqref{phase space symplectic}. Then $E:\huaE\oplus\huaE^*\rightarrow\huaE\oplus\huaE^*$ defined by
\begin{eqnarray*}
E(x+\xi)=x-\xi,\,\,\,\forall x\in\huaE,\xi\in\huaE^*
\end{eqnarray*}
is a paracomplex structure and $(\huaE\oplus\huaE^*,\huaB,E)$ is a para-K\"{a}hler Leibniz algebra.
\end{pro}

\begin{proof}
It is directly deduced by Theorem \ref{para-Kahler if and only if}.
\end{proof}

Let $(\huaE,\huaB,E)$ be a para-K\"{a}hler Leibniz algebra. Then it is obvious that $\huaE_-$ is isomorphic to $\huaE^*_+$ via the symplectic structure $\huaB$. Moreover it is straightforward to deduce that

\begin{pro}
Any para-K\"{a}hler Leibniz algebra $(\huaE,\huaB,E)$ is isomorphic to the para-K\"{a}hler Leibniz algebra associated to a phase space $(\huaE_+\oplus\huaE^*_+,\huaB)$ of the Leibniz algebra $\huaE_+$.
\end{pro}

\begin{defi}
A {\bf pseudo-Riemannian Leibniz algebra} is a Leibniz algebra $(\huaE,[\cdot,\cdot]_\huaE)$ endowed with a nondegenerate skew-symmetric bilinear form $S$. The associated {\bf Levi-Civita products}   $\ast,\star:\huaE\otimes\huaE\rightarrow\huaE$ are defined by the following formulas:
\begin{eqnarray}
2S(x\ast y,z)&=&S([x,y]_\huaE,z)+S([y,z]_\huaE,x)+S([z,y]_\huaE,x)+S([x,z]_\huaE,y),\label{Levi-Civita1}\\
2S(x\star y,z)&=&S([x,y]_\huaE,z)-S([y,z]_\huaE,x)-S([z,y]_\huaE,x)-S([x,z]_\huaE,y).\label{Levi-Civita2}
\end{eqnarray}
\end{defi}

\begin{pro}
Let $(\huaE,S)$ be a pseudo-Riemannian Leibniz algebra. Then the Levi-Civita products $(\ast,\star)$ satisfy the following equalities:
\begin{eqnarray}
x\ast y+x\star y&=&[x,y]_\huaE,\label{Levi-Civita property1}\\
S(x\ast y,z)+S(y,x\ast z)&=&0,\label{Levi-Civita property2}\\
S(x\star y,z)-S(x,y\ast z+z\star y)&=&0.\label{Levi-Civita property3}
\end{eqnarray}
\end{pro}

\begin{proof}
By \eqref{Levi-Civita1} and \eqref{Levi-Civita2}, we have
\begin{eqnarray*}
2S(x\ast y+x\star y,z)=2S([x,y]_\huaE,z),
\end{eqnarray*}
which implies that $x\ast y+x\star y=[x,y]_\huaE$.
Since $S$ is skew-symmetric, we have
\begin{eqnarray*}
2S(y,x\ast z)&=&-2S(x\ast z,y)=-S([x,z]_\huaE,y)-S([z,y]_\huaE,x)-S([y,z]_\huaE,x)-S([x,y]_\huaE,z),\\
2S(x,y\ast z)&=&-2S(y\ast z,x)=-S([y,z]_\huaE,x)-S([z,x]_\huaE,y)-S([x,z]_\huaE,y)-S([y,x]_\huaE,z),\\
2S(x,z\star y)&=&-2S(z\star y,x)=-S([z,y]_\huaE,x)+S([y,x]_\huaE,z)+S([x,y]_\huaE,z)+S([z,x]_\huaE,y),
\end{eqnarray*}
which shows that $S(x\ast y,z)+S(y,x\ast z)=0$ and $S(x\star y,z)-S(x,y\ast z+z\star y)=0$.
\end{proof}

\begin{pro}
Let $(\huaE,\huaB,E)$ be a para-K\"{a}hler Leibniz algebra. Then $(\huaE,S)$ is a pseudo-Riemannian Leibniz algebra, where $S$ is defined by
\begin{eqnarray}
\label{para-Kahler to pseudo}S(x,y)=\huaB(x,Ey),\,\,\,\forall x,y\in\huaE.
\end{eqnarray}
\end{pro}

\begin{proof}
Since $\huaB$ is symmetric and $\huaB(Ex,Ey)=-\huaB(x,y)$, we have
$$S(y,x)=\huaB(y,Ex)=-\huaB(Ey,x)=-\huaB(x,Ey)=-S(x,y),$$
which implies that $S$ is skew-symmetric. Moreover, since $\huaB$ is nondegenerate and $E^2=\Id$, it is obvious that $S$ is nondegenerate.
\end{proof}

The following proposition clarifies the relationship between the Levi-Civita products and the Leibniz-dendriform algebra structure on a para-K\"{a}hler Leibniz algebra.

\begin{pro}
Let $(\huaE,\huaB,E)$ be a para-K\"{a}hler Leibniz algebra and $(\ast,\star)$ the associated Levi-Civita products. Then for all $x_1,x_2\in\huaE_+$ and $\xi_1,\xi_2\in\huaE_-$, we have
\begin{eqnarray*}
&&x_1\ast x_2=x_1\lhd x_2,\,\,\,x_1\star x_2=x_1\rhd x_2,\\
&&\,\xi_1\ast \xi_2=\xi_1\lhd \xi_2,\,\,\,\,\xi_1\star \xi_2=\xi_1\rhd \xi_2,
\end{eqnarray*}
where $\lhd,\rhd$ is given by \eqref{sym-to-LeiDen1} and \eqref{sym-to-LeiDen2}.
\end{pro}

\begin{proof}
Since $(\huaE,\huaB,E)$ is a para-K\"{a}hler Leibniz algebra, Leibniz subalgebras $\huaE_+$ and $\huaE_-$ are isotropic. For all $x_1,x_2,x_3\in\huaE_+$, we have
\begin{eqnarray*}
&&2\huaB(x_1\ast x_2,x_3)\\
&=&2\huaB(x_1\ast x_2,Ex_3)=2S(x_1\ast x_2,x_3)\\
&=&S([x_1,x_2]_\huaE,x_3)+S([x_2,x_3]_\huaE,x_1)+S([x_3,x_2]_\huaE,x_1)+S([x_1,x_3]_\huaE,x_2)\\
&=&\huaB([x_1,x_2]_\huaE,x_3)+\huaB([x_2,x_3]_\huaE,x_1)+\huaB([x_3,x_2]_\huaE,x_1)+\huaB([x_1,x_3]_\huaE,x_2)\\
&=&0.
\end{eqnarray*}
By $(\huaE_+)^\perp=\huaE_+$, we obtain $x_1\ast x_2\in\huaE_+$. Similarly, $x_1\star x_2\in\huaE_+$ and
$\xi_1\ast \xi_2,\xi_1\star \xi_2\in\huaE_-$ for all $\xi_1,\xi_2\in\huaE_-$. Furthermore, for all $x_1,x_2\in\huaE_+$ and $\xi\in\huaE_-$, we have
\begin{eqnarray*}
&&2\huaB(x_1\ast x_2,\xi)\\
&=&-2\huaB(x_1\ast x_2,E\xi)=-2S(x_1\ast x_2,\xi)\\
&=&-S([x_1,x_2]_\huaE,\xi)-S([x_2,\xi]_\huaE,x_1)-S([\xi,x_2]_\huaE,x_1)-S([x_1,\xi]_\huaE,x_2)\\
&=&\huaB([x_1,x_2]_\huaE,\xi)-\huaB([x_2,\xi]_\huaE,x_1)-\huaB([\xi,x_2]_\huaE,x_1)-\huaB([x_1,\xi]_\huaE,x_2)\\
&=&-\huaB([x_1,\xi]_\huaE,x_2)-\huaB([x_1,\xi]_\huaE,x_2)\\
&=&-2\huaB(x_2,[x_1,\xi]_\huaE)\\
&=&2\huaB(x_1\lhd x_2,\xi),
\end{eqnarray*}
which implies that $x_1\ast x_2=x_1\lhd x_2$. Similarly, we can prove the rest in a similar way.
\end{proof}

\begin{pro}\label{lei-den to para-Kahler}
Let $(\huaL,\lhd,\rhd)$ be a Leibniz-dendriform algebra. Then $(\huaL\ltimes_{L^*_\lhd,-L^*_\lhd-R^*_\rhd}\huaL^*,\huaB,E)$ is a para-K\"{a}hler Leibniz algebra, where $\huaB$ is given by \eqref{phase space symplectic} and $E$ is given by \eqref{paracomplex str}.
Furthermore, for all $x\in\huaL$ and $\xi\in\huaL^*$, we have
$$x\ast \xi=L^*_{\lhd x} \xi,\,\,\,\xi\ast x=0,\,\,\,
x\star \xi=0,\,\,\,\xi\star x=(-L^*_{\lhd x}-R^*_{\rhd x})\xi.$$
\end{pro}

\begin{proof}
By Theorem \ref{phase space and Lei-Den}, $(\huaL\ltimes_{L^*_\lhd,-L^*_\lhd-R^*_\rhd}\huaL^*,\huaB)$ is a phase space of the Leibniz algebra $(\huaL,[\cdot,\cdot]_{\lhd,\rhd})$.
Then $(\huaL\ltimes_{L^*_\lhd,-L^*_\lhd-R^*_\rhd}\huaL^*,\huaB,E)$ is a para-K\"{a}hler Leibniz algebra by Proposition \ref{phase space to para-Kahler} and $S(\huaL,\huaL)=S(\huaL^*,\huaL^*)=0$ by \eqref{para-Kahler to pseudo}. For all $x,y\in\huaL$ and $\xi,\eta\in\huaL^*$, we have
\begin{eqnarray*}
2\huaB(x\ast \xi,\eta)&=&-2\huaB(x\ast \xi,E\eta)=-2S(x\ast \xi,\eta)\\
&=&-S([x,\xi]_\ltimes,\eta)-S([\xi,\eta]_\ltimes,x)-S([\eta,\xi]_\ltimes,x)-S([x,\eta]_\ltimes,\xi)\\
&=&0,
\end{eqnarray*}
which implies that $x\ast \xi\in\huaL^*$. Then we have
\begin{eqnarray*}
2\huaB(x\ast \xi,y)&=&2\huaB(x\ast \xi,Ey)=2S(x\ast \xi,y)\\
&=&S([x,\xi]_\ltimes,y)+S([\xi,y]_\ltimes,x)+S([y,\xi]_\ltimes,x)+S([x,y]_\ltimes,\xi)\\
&=&\huaB([x,\xi]_\ltimes,y)+\huaB([\xi,y]_\ltimes,x)+\huaB([y,\xi]_\ltimes,x)-\huaB([x,y]_\ltimes,\xi)\\
&=&\huaB(L^*_{\lhd x}\xi,y)+\huaB((-L^*_{\lhd y}-R^*_{\rhd y})\xi,x)
+\huaB(L^*_{\lhd y}\xi,x)-\huaB([x,y]_{\lhd,\rhd},\xi)\\
&=&-\langle\xi,x\lhd y\rangle+\langle\xi,y\lhd x+x\rhd y\rangle-\langle\xi, y\lhd x\rangle-\langle x\lhd y+x\rhd y,\xi\rangle\\
&=&2\langle L^*_{\lhd x}\xi,y\rangle\\
&=&2\huaB(L^*_{\lhd x}\xi,y),
\end{eqnarray*}
which implies $x\ast \xi=L^*_{\lhd x} \xi$. Similarly, we obtain $\xi\star x=(-L^*_{\lhd x}-R^*_{\rhd x})\xi$. Since
\begin{eqnarray*}
2\huaB(\xi\ast x,y)&=&2\huaB(\xi\ast x,Ey)=2S(\xi\ast x,y)\\
&=&S([\xi,x]_\ltimes,y)+S([x,y]_\ltimes,\xi)+S([y,x]_\ltimes,\xi)+S([\xi,y]_\ltimes,x)\\
&=&\huaB([\xi,x]_\ltimes,y)-\huaB([x,y]_\ltimes,\xi)-\huaB([y,x]_\ltimes,\xi)+\huaB([\xi,y]_\ltimes,x)\\
&=&\huaB((-L^*_{\lhd x}-R^*_{\rhd x})\xi,y)-\huaB([x,y]_{\lhd,\rhd},\xi)
-\huaB([y,x]_{\lhd,\rhd},\xi)+\huaB((-L^*_{\lhd y}-R^*_{\rhd y})\xi,x)\\
&=&\langle \xi,x\lhd y+y\rhd x\rangle-\langle x\lhd y+x\rhd y,\xi\rangle-\langle y\lhd x+y\rhd x,\xi\rangle+\langle \xi,y\lhd x+x\rhd y\rangle\\
&=&0,\\
2\huaB(\xi\ast x,\eta)&=&-2\huaB(\xi\ast x,E\eta)=-2S(\xi\ast x,\eta)\\
&=&-S([\xi,x]_\ltimes,\eta)-S([x,\eta]_\ltimes,\xi)-S([\eta,x]_\ltimes,\xi)-S([\xi,\eta]_\ltimes,x)=0,
\end{eqnarray*}
we get $\xi\ast x\in\huaL$ and $\xi\ast x\in\huaL^*$, which implies that $\xi\ast x=0$. We obtain $x\star\xi=0$ in the same way.
\end{proof}

\section{Complex product structures on Leibniz algebras}\label{sec:compro}

In this section, we introduce the notion of a complex product
structure on a Leibniz algebra. We   construct complex product structures using Leibniz-dendriform algebras and conversely, a complex product
structure also induces certain Leibniz-dendriform algebra structures.
\begin{pro}\label{relation between product and complex str}
Let $\huaE$ be a complex Leibniz algebra. Then $E$ is a product structure on $\huaE$ if and only if $J=iE$ is a complex structure on $\huaE$.
\end{pro}

\begin{proof}
Let $J=iE$ be a complex structure on $\huaE$. We have $E^2=-J^2=\Id$, which means that $E$ is an almost product structure on $\huaE$. By   \eqref{complex str}, we have
\begin{eqnarray*}
E[x,y]_\huaE&=&-iJ[x,y]_\huaE\\
&=&-i[Jx,y]_{\huaE}-i[x,Jy]_{\huaE}-iJ[Jx,Jy]_{\huaE}\\
&=&[-iJx,y]_{\huaE}+[x,-iJy]_{\huaE}-(-iJ)[-iJx,-iJy]_{\huaE}\\
&=&[Ex,y]_{\huaE}+[x,Ey]_{\huaE}-E[Ex,Ey]_{\huaE}.
\end{eqnarray*}
Thus, $E$ is a product structure on the complex Leibniz algebra $\huaE$.

The other part can be proved in a similar way and we leave it to readers.
\end{proof}

\begin{cor}\label{complex to paracomplex}
Let $J$ be a complex structure on the real Leibniz algebra $(\huaE,[\cdot,\cdot]_\huaE)$. Then $-iJ_{\mathds{C}}$ is a paracomplex structure on the complex Leibniz algebra $(\huaE_{\mathds{C}},[\cdot,\cdot]_{\huaE_{\mathds{C}}})$, where $J_{\mathds{C}}$ is defined by \eqref{JC_1}.
\end{cor}

\begin{proof}
By Theorem \ref{complex to decomposition}, we have $\huaE_{\mathds{C}}=\huaE_{i}\oplus\huaE_{-i}$ and $\huaE_{i}=\sigma(\huaE_{-i})$, where $\huaE_{i}$ and $\huaE_{-i}$ are subalgebras of $\huaE_{\mathds{C}}$. It is obvious that dim($\huaE_{i}$)=dim($\huaE_{-i}$).
%By Theorem \ref{product to decomposition} and Definition \ref{def of paracomplex str}, there is a paracomplex structure on complex Leibniz algebra $\huaE_{\mathds{C}}$.
By the proof of Theorem \ref{complex to decomposition}, $J_{\mathds{C}}$ is a complex structure on $\huaE_{\mathds{C}}$. Then $-iJ_{\mathds{C}}$ is a product structure on the complex Leibniz algebra $\huaE_{\mathds{C}}$ by Proposition \ref{relation between product and complex str}. Furthermore, $\huaE_{i}$ and $\huaE_{-i}$ are eigenspaces of $-iJ_{\mathds{C}}$ corresponding to $+1$ and $-1$, which implies that $-iJ_{\mathds{C}}$ is a paracomplex structure.
\end{proof}

\begin{defi}
Let $(\huaE,[\cdot,\cdot]_\huaE)$ be a real Leibniz algebra. A {\bf complex product structure} on the Leibniz algebra $\huaE$ is a pair $\{J,E\}$ of a complex structure $J$ and a product structure $E$ satisfying
\begin{eqnarray}\label{complex product str}
J\circ E=-E\circ J.
\end{eqnarray}
\end{defi}

\begin{rmk}\label{complex product str to paracomplex str}
Let $\{J,E\}$ be a complex product structure on the real Leibniz algebra $(\huaE,[\cdot,\cdot]_\huaE)$. For all $x\in\huaE_+$, we have $E(J(x))=-J(E(x))=-J(x)$ by \eqref{complex product str}, which implies that $J(\huaE_+)\subset\huaE_-$. Similarly, we have $J(\huaE_-)\subset\huaE_+$. Thus, we obtain $J(\huaE_+)=\huaE_-$ and $J(\huaE_-)=\huaE_+$. Therefore, {\rm dim(}$\huaE_{+}${\rm )}={\rm dim(}$\huaE_{-}${\rm )} and then $E$ is a paracomplex structure on $\huaE$.
\end{rmk}

\begin{thm}
Let $(\huaE,[\cdot,\cdot]_\huaE)$ be a real Leibniz algebra. Then $\huaE$ has a complex product structure if and only if $\huaE$ has a complex structure $J$ and has a vector space decomposition $\huaE=\huaE_{+}\oplus\huaE_{-}$, where $\huaE_+,\,\huaE_-$ are Leibniz subalgebras of $\huaE$ and $J(\huaE_+)=\huaE_-$.
\end{thm}

\begin{proof}
Let $\{J,E\}$ be a complex product structure on $\huaE$ and let $\huaE_{\pm}$ denote the eigenspaces corresponding to the eigenvalues $\pm 1$ of $E$. Then $\huaE_+$ and $\huaE_-$ are subalgebras of $\huaE$ by Theorem \ref{product to decomposition}. By Remark \ref{complex product str to paracomplex str}, we have $J(\huaE_+)=\huaE_-$.

Conversely, by Theorem \ref{product to decomposition}, we obtain a product structure $E$ on $\huaE$ defined as \eqref{decomposition to product}. By $J(\huaE_+)=\huaE_-$ and $J^2=-\Id$, we have
$$E(J(x+\xi))=E(J(x)+J(\xi))=-J(x)+J(\xi)=-J(x-\xi)=-J(E(x+\xi)),\,\,\,\forall x\in\huaE_+,\,\xi\in\huaE_-.$$
Thus, $\{J,E\}$ is a complex product structure on $\huaE$.
\end{proof}

\begin{pro}\label{complex product pro}
Let $E$ be a paracomplex structure on the real Leibniz algebra $(\huaE,[\cdot,\cdot]_\huaE)$. Then there is a complex product structure $\{J,E\}$ on $\huaE$ if and only if there exists a linear isomorphism $\phi:\huaE_+\rightarrow\huaE_-$ satisfying the following equality:
\begin{eqnarray}\label{complex product str and phi}
\phi[x_1,x_2]_\huaE=[\phi (x_1),x_2]_\huaE+[x_1,\phi (x_2)]_\huaE-\phi^{-1}[\phi (x_1),\phi (x_2)]_\huaE,\,\,\,\forall x_1,x_2\in\huaE_+.
\end{eqnarray}
\end{pro}

\begin{proof}
Let $\{J,E\}$ be a complex product structure on $\huaE$. Define a linear isomorphism $\phi:\huaE_+\rightarrow\huaE_-$ by ${\phi\triangleq} J\mid_{\huaE_+}$. It is straightforward to show that $\phi$ satisfies \eqref{complex product str and phi} by $J^2=-\Id$ and  \eqref{complex str}.

Conversely, define an endomorphism $J$ of $\huaE$ by
\begin{eqnarray}\label{phi to complex str}
J(x+\xi)=-\phi^{-1}(\xi)+\phi(x),\,\,\,\forall x\in\huaE_+,\,\xi\in\huaE_-.
\end{eqnarray}
Then we have $J^2=-\Id$ and $J\circ E=-E\circ J$ by direct computation. For any $\xi_1,\xi_2\in\huaE_-$, let $x_1,x_2\in\huaE_+$ such that $\phi(x_1)=\xi_1,\phi(x_2)=\xi_2$. By \eqref{complex product str and phi}, we have
\begin{eqnarray*}
&&[J\xi_1,\xi_2]_\huaE+[\xi_1,J\xi_2]_\huaE+J[J\xi_1,J\xi_2]_\huaE\\
&=&[-\phi^{-1}(\xi_1),\xi_2]_\huaE+[\xi_1,-\phi^{-1}(\xi_1)]_\huaE+\phi[-\phi^{-1}(\xi_1),-\phi^{-1}(\xi_2)]_\huaE\\
&=&-[x_1,\phi(x_2)]_\huaE-[\phi(x_1),x_2]_\huaE+\phi[x_1,x_2]_\huaE\\
&=&-\phi^{-1}[\phi (x_1),\phi (x_2)]_\huaE\\
&=&J[\xi_1,\xi_2]_\huaE,
\end{eqnarray*}
which implies that \eqref{complex str} holds for all $\xi_1,\xi_2\in\huaE_-$. Similarly, we can deduce that \eqref{complex str} holds for all the other cases. Thus, $J$ is a complex structure and $\{J,E\}$ is a complex product structure on $\huaE$.
\end{proof}

\begin{defi}
Let $(\huaL,\lhd,\rhd)$ be a Leibniz-dendriform algebra. A nondegenerate bilinear form $\omega\in\huaL^*\otimes\huaL^*$ is called {\bf invariant} if for any $x,y,z\in\huaL$, the following conditions hold:
\begin{eqnarray}
\omega(x\lhd y,z)&=&-\omega(y,x\lhd z),\label{invariant1}\\
\omega(x\rhd y,z)&=&\omega(x,y\lhd z+z\rhd y).\label{invariant2}
\end{eqnarray}
\end{defi}
\emptycomment{
Define $\phi:\huaL\rightarrow\huaL^*$ by $\omega(x,y)=<\phi(x),y>$, then the above definition shows that $\phi$ is a representation homomorphism from $(\huaL;L_\lhd,R_\rhd)$ to $(\huaL^*;L^*_\lhd,-L^*_\lhd-R^*_\rhd)$}

A nondegenerate bilinear form $\omega$ naturally induces a linear isomorphism
$\omega^{\sharp}:\huaL\rightarrow\huaL^*$  by
\begin{eqnarray*}
\langle\omega^{\sharp}(x),y\rangle=\omega(x,y),\,\,\,\forall x,y\in\huaL.
\end{eqnarray*}

\begin{pro}\label{phase space complex product}
Let $(\huaL,\lhd,\rhd)$ be a Leibniz-dendriform algebra with a nondegenerate invariant bilinear form $\omega$. Then there is a complex product structure $\{J,E\}$ on the semidirect product Leibniz algebra $\huaL\ltimes_{L^*_\lhd,-L^*_\lhd-R^*_\rhd}\huaL^*$, where the product structure $E$ is given by \eqref{paracomplex str} and the complex structure $J$ is given by
\begin{eqnarray}\label{lei-den invariant to complex str}
J(x+\xi)=-\omega^{\sharp -1}(\xi)+\omega^{\sharp}(x),\,\,\,\forall x\in\huaL,\xi\in\huaL^*.
\end{eqnarray}
\end{pro}

\begin{proof}
By Proposition \ref{paracomplex pro}, $E$ is a paracomplex structure. For any $x,y\in\huaL$, we have
\begin{eqnarray*}
&&[\omega^{\sharp}(x),y]_\ltimes+[x,\omega^{\sharp}(y)]_\ltimes-\omega^{\sharp -1}[\omega^{\sharp}(x),\omega^{\sharp}(y)]_\ltimes\\
&=&[\omega^{\sharp}(x),y]_\ltimes+[x,\omega^{\sharp}(y)]_\ltimes\\
&=&L^*_{\lhd x}\omega^{\sharp}(y)-L^*_{\lhd y}\omega^{\sharp}(x)-R^*_{\rhd y}\omega^{\sharp}(x).
\end{eqnarray*}
By \eqref{invariant1} and \eqref{invariant2}, we have
\begin{eqnarray*}
\langle \omega^{\sharp}[x,y]_{\lhd,\rhd},z\rangle
&=&\langle \omega^{\sharp}(x\lhd y),z\rangle +\langle \omega^{\sharp}(x\rhd y),z\rangle
=\omega(x\lhd y,z)+\omega(x\rhd y,z)\\
&=&-\omega(y,x\lhd z)+\omega(x,y\lhd z+z\rhd y)\\
&=&-\omega(y,L_{\lhd x}z)+\omega(x,L_{\lhd y}z+R_{\rhd y}z)\\
&=&-\langle \omega^{\sharp}(y),L_{\lhd x}z\rangle +\langle \omega^{\sharp}(x),L_{\lhd y}z+R_{\rhd y}z\rangle \\
&=&\langle L^*_{\lhd x}\omega^{\sharp}(y),z\rangle -\langle (L^*_{\lhd y}+R^*_{\rhd y})\omega^{\sharp}(x),z\rangle \\
&=&\langle L^*_{\lhd x}\omega^{\sharp}(y)-L^*_{\lhd y}\omega^{\sharp}(x)-R^*_{\rhd y}\omega^{\sharp}(x),z\rangle ,
\end{eqnarray*}
which implies that
\begin{eqnarray*}
\omega^{\sharp}[x,y]_{\lhd,\rhd}=L^*_{\lhd x}\omega^{\sharp}(y)-L^*_{\lhd y}\omega^{\sharp}(x)-R^*_{\rhd y}\omega^{\sharp}(x).
\end{eqnarray*}
Thus, we have
\begin{eqnarray*}
\omega^{\sharp}[x,y]_{\lhd,\rhd}=
[\omega^{\sharp}(x),y]_\ltimes+[x,\omega^{\sharp}(y)]_\ltimes-\omega^{\sharp -1}[\omega^{\sharp}(x),\omega^{\sharp}(y)]_\ltimes,
\end{eqnarray*}
which implies that $\{J,E\}$ is a complex product structure on $\huaL\ltimes_{L^*_\lhd,-L^*_\lhd-R^*_\rhd}\huaL^*$ by Proposition \ref{complex product pro}.
\end{proof}

\emptycomment{\color{red}
\begin{cor}\label{invariant to complex cor}
In particularly, let $\omega$ be skew-symmetric,  $\{e_1,\cdot\cdot\cdot,e_n,f_1,\cdot\cdot\cdot,f_n\}$ a basis of the Leibniz-dendriform algebra $\huaL$ such that $\omega(e_i,e_j)=\omega(f_i,f_j)=0$, $\omega(e_i,f_j)=\delta_{ij}$ and $\{e^*_1,\cdot\cdot\cdot,e^*_n,f^*_1,\cdot\cdot\cdot,f^*_n\}$ the dual basis, then the complex structure $J$ defined by a nondegenerate invariant bilinear form $\omega$ on $\huaL\ltimes_{L^*_\lhd,-L^*_\lhd-R^*_\rhd}\huaL^*$ is given as follows:
\begin{eqnarray*}
J(e_i)=f^*_i,\,\,\,J(f_j)=-e^*_j,\,\,\,J(e^*_i)=f_i,\,\,\,J(f^*_j)=-e_j.
\end{eqnarray*}
\end{cor}
}
The following proposition shows that it is impossible to find
interesting examples of complex product structures if the complex structure is strict.

\begin{pro}
Let $\{J,E\}$ be a complex product structure on the real Leibniz algebra $(\huaE,[\cdot,\cdot]_\huaE)$, where $J$ is a strict complex structure.
Then $\huaE$ is abelian.
\end{pro}

\begin{proof}
There is a decomposition $\huaE=\huaE_{+}\oplus\huaE_{-}$, where $\huaE_+,\,\huaE_-$ are Leibniz subalgebras of $\huaE$ and $J(\huaE_+)=\huaE_-$. Since $J$ is a strict complex structure, we have $[x,y]_\huaE=-[Jx,Jy]_\huaE$ for all $x,y\in\huaE$ by \eqref{new bracket} and Proposition \ref{strict complex equvialent}. Then we have $[\huaE_+,\huaE_+]_\huaE=-[\huaE_-,\huaE_-]_\huaE$, which implies that $\huaE_+$ and $\huaE_-$ are abelian. For all $x_1,x_2\in\huaE_+$, by \eqref{complex condition 1} and \eqref{complex condition 2}, we have
\begin{eqnarray*}
E[x_1,J(x_2)]_\huaE&=&EJ[x_1,x_2]_\huaE=-JE[x_1,x_2]_\huaE
=-J[x_1,x_2]_\huaE=-[x_1,J(x_2)]_\huaE,\\
E[x_1,J(x_2)]_\huaE&=&E[-J^2(x_1),J(x_2)]_\huaE=-EJ[J(x_1),J(x_2)]_\huaE=JE[J(x_1),J(x_2)]_\huaE\\
&=&-J[J(x_1),J(x_2)]_\huaE=[x_1,J(x_2)]_\huaE,
\end{eqnarray*}
which implies that $[x_1,J(x_2)]_\huaE\in\huaE_-$ and $[x_1,J(x_2)]_\huaE\in\huaE_+$ respectively. Thus $[\huaE_+,\huaE_-]_\huaE=0$. Similarly, we have $[\huaE_-,\huaE_+]_\huaE=0$. Therefore, $\huaE$ is abelian.
\end{proof}

At the end of this section, we will show that a complex product structure can induce certain Leibniz-dendriform algebra structures.
\begin{pro}
Let $\{J,E\}$ be a complex product structure on the real Leibniz algebra $(\huaE,[\cdot,\cdot]_\huaE)$, where $\huaE=\huaE_{+}\oplus\huaE_{-}$ as vector spaces, $\huaE_{+}$ and $\huaE_{-}$ are Leibniz subalgebras of $\huaE$ and $J(\huaE_+)=\huaE_-$. Then $\huaE_+$ and $\huaE_-$ carry  compatible Leibniz-dendriform algebra structures respectively.
\end{pro}

\begin{proof}
For any $x\in\huaE_+$ and $\xi\in\huaE_-$, define ${l^+},{r^+}:\huaE_+\rightarrow\gl(\huaE_-)$ and ${l^-},{r^-}:\huaE_-\rightarrow\gl(\huaE_+)$ by:
\begin{eqnarray*}
\label{eq:decompositon to matched pair1}[x,\xi]_\huaE&=&r^-_\xi(x)+l^+_x(\xi),\\
\label{eq:decompositon to matched pair2}{[\xi,x]}_\huaE&=&l^-_\xi(x)+r^+_x(\xi).
\end{eqnarray*}
Then we define $\textswab{l}^+,\textswab{r}^+:\huaE_+\rightarrow\gl(\huaE_+)$ and $\textswab{l}^-,\textswab{r}^-:\huaE_-\rightarrow\gl(\huaE_-)$ as follows:
\begin{eqnarray*}
&&\textswab{l}^+_x=-J{l^+_x}J,\,\,\,\textswab{r}^+_x=-J{r^+_x}J;\\
&&\textswab{l}^-_\xi=-J{l^-_\xi}J,\,\,\,\textswab{r}^-_\xi=-J{r^-_\xi}J.
\end{eqnarray*}
By direct computation we obtain that $(\huaE_+;\textswab{l}^+,\textswab{r}^+)$ is a representations of $\huaE_+$ and $(\huaE_-;\textswab{l}^-,\textswab{r}^-)$ is a representations of $\huaE_-$.

Considering the projections $\pi_+:\huaE\rightarrow\huaE_+$ and $\pi_-:\huaE\rightarrow\huaE_-$, we have
\begin{eqnarray*}
&&\textswab{l}^+_{x_1}(x_2)=-{\pi_+}J[x_1,J(x_2)]_\huaE,\,\,\,
\textswab{r}^+_{x_1}(x_2)=-{\pi_+}J[J(x_2),x_1]_\huaE;\\
&&\textswab{l}^-_{\xi_1}(\xi_2)=-{\pi_-}J[\xi_1,J(\xi_2)]_\huaE,\,\,\,\,\,
\textswab{r}^-_{\xi_1}(\xi_2)=-{\pi_-}J[J(\xi_2),\xi_1]_\huaE,
\end{eqnarray*}
for any $x_1,x_2\in\huaE_+$ and $\xi_1,\xi_2\in\huaE_-$. It is easy to see that $\Id:\huaE_+\rightarrow\huaE_+$ is an invertible Rota-Baxter operator associated to $(\huaE_+;\textswab{l}^+,\textswab{r}^+)$ and $\Id:\huaE_-\rightarrow\huaE_-$ is an invertible Rota-Baxter operator associated to $(\huaE_-;\textswab{l}^-,\textswab{r}^-)$. By Proposition \ref{o to Lei-Den}, the compatible Leibniz-dendriform algebra structures on $\huaE_+$ and $\huaE_-$ are respectively given by
\begin{eqnarray*}
&&x_1\lhd^+ x_2=\textswab{l}^+_{x_1}(x_2),\,\,\,
x_1\rhd^+ x_2=\textswab{r}^+_{x_2}(x_1);\\
&&\xi_1\lhd^- \xi_2=\textswab{l}^-_{\xi_1}(\xi_2),\,\,\,\,\,
\xi_1\rhd^- \xi_2=\textswab{r}^-_{\xi_2}(\xi_1).
\end{eqnarray*}
for all $x_1,x_2\in\huaE_+$ and $\xi_1,\xi_2\in\huaE_-.$
\end{proof}

\section{Pseudo-K\"{a}hler structures on Leibniz algebras}\label{sec:symcom}

In this section, we introduce the notion of a pseudo-K\"{a}hler structure on a Leibniz algebra by adding a compatibility condition between a symplectic structure and a complex structure on the Leibniz algebra. Furthermore, we investigate the relation between para-K\"{a}hler structures and pseudo-K\"{a}hler structures on a Leibniz algebra.

\begin{defi}
Let $\huaB$ be a symplectic structure and $J$ a complex structure on the real Leibniz algebra $(\huaE,[\cdot,\cdot]_\huaE)$. The triple $(\huaE,\huaB,J)$ is called a real {\bf pseudo-K\"{a}hler} Leibniz algebra if
\begin{eqnarray}
\huaB(Jx,Jy)=\huaB(x,y),\,\,\,\forall x,y\in\huaE.\label{pseudo-Kahler str}
\end{eqnarray}
\end{defi}

\begin{pro}
Let $(\huaE,\huaB,J)$ be a real pseudo-K\"{a}hler Leibniz algebra. Then $(\huaE,S)$ is a pseudo-Riemannian Leibniz algebra, where $S$ is defined by
\begin{eqnarray}\label{pseudo-Kahler to pseudo-Riemannian}
S(x,y)=\huaB(x,Jy),\,\,\,\forall x,y\in\huaE.
\end{eqnarray}
\end{pro}

\begin{proof}
By \eqref{pseudo-Kahler str}, we have
$$S(y,x)=\huaB(y,Jx)=-\huaB(Jy,x)=-\huaB(x,Jy)=-S(x,y),$$
which implies that $S$ is skew-symmetric. Moreover, since $\huaB$ is nondegenerate and $J^2=-\Id$, it is obvious that $S$ is nondegenerate. Then $(\huaE,S)$ is a pseudo-Riemannian Leibniz algebra.
\end{proof}

\begin{thm}
Let $(\huaE,\huaB,E)$ be a complex para-K\"{a}hler Leibniz algebra. Then $(\huaE_{\mathds{R}},\huaB_{\mathds{R}},J)$ is a real pseudo-K\"{a}hler Leibniz algebra, where $\huaE_{\mathds{R}}$ is the underlying real Leibniz algebra, $J=iE$ and $\huaB_{\mathds{R}}={\rm Re}(\huaB)$ is the real part of $\huaB$.
\end{thm}

\begin{proof}
We obtain that $J=iE$ is a complex structure on the complex Leibniz algebra $\huaE$ by Proposition \ref{relation between product and complex str}.
Thus, $J$ is also a complex structure on the real Leibniz algebra $\huaE_{\mathds{R}}$. It is obvious that $\huaB_{\mathds{R}}$ is symmetric. If for all $x\in\huaE,~\huaB_{\mathds{R}}(x,y)=0$, then we have
$$\huaB(x,y)=\huaB_{\mathds{R}}(x,y)+i\huaB_{\mathds{R}}(-ix,y)=0.$$
Since $\huaB$ is nondegenerate, we obtain $y=0$, which shows that $\huaB_{\mathds{R}}$ is nondegenerate. Therefore, $\huaB_{\mathds{R}}$ is a symplectic structure on the real Leibniz algebra $\huaE_{\mathds{R}}$. By \eqref{para-Kahler str}, we have
$$\huaB_{\mathds{R}}(Jx,Jy)={\rm Re}(\huaB(iEx,iEy))={\rm Re}(-\huaB(Ex,Ey))
={\rm Re}(\huaB(x,y))=\huaB_{\mathds{R}}(x,y).$$
Thus, $(\huaE_{\mathds{R}},\huaB_{\mathds{R}},iE)$ is a real pseudo-K\"{a}hler Leibniz algebra.
\end{proof}

Conversely, we have
\begin{thm}
Let $(\huaE,\huaB,J)$ be a real pseudo-K\"{a}hler Leibniz algebra. Then $(\huaE_{\mathds{C}},\huaB_{\mathds{C}},E)$ is a complex para-K\"{a}hler Leibniz algebra, where $\huaE_{\mathds{C}}=\huaE\otimes_{\mathds{R}}{\mathds{C}}$ is the complexification of $\huaE$, $E=-iJ_{\mathds{C}}$, $J_{\mathds{C}}$ is defined by \eqref{JC_1} and $\huaB_{\mathds{C}}$ is the complexification of $\huaB$, that is,
$$\huaB_{\mathds{C}}(x+iy,z+iw)=\huaB(x,z)-\huaB(y,w)+i\huaB(x,w)+i\huaB(y,z),\,\,\,\forall x,y,z,w\in\huaE.$$
\end{thm}

\begin{proof}
By Corollary \ref{complex to paracomplex}, $E=-iJ_{\mathds{C}}$ is a paracomplex structure on the complex Leibniz algebra $\huaE_{\mathds{C}}$. It is obvious that $\huaB_{\mathds{C}}$ is symmetric and nondegenerate. Moreover, since $\huaB$ is a symplectic structure on $\huaE$, we deduce that $\huaB_{\mathds{C}}$ is a symplectic structure on $\huaE_{\mathds{C}}$. Finally, by \eqref{pseudo-Kahler str}, we have
\begin{eqnarray*}
\huaB_{\mathds{C}}(E(x+iy),E(z+iw))&=&\huaB_{\mathds{C}}(Jy-iJx,Jw-iJz)\\
&=&\huaB(Jy,Jw)-\huaB(Jx,Jz)-i\huaB(Jy,Jz)-i\huaB(Jx,Jz)\\
&=&\huaB(y,w)-\huaB(x,z)-i\huaB(y,z)-i\huaB(x,z)\\
&=&-\huaB_{\mathds{C}}(x+iy,z+iw).
\end{eqnarray*}
Therefore, $(\huaE_{\mathds{C}},\huaB_{\mathds{C}},-iJ_{\mathds{C}})$ is a complex para-K\"{a}hler Leibniz algebra.
\end{proof}

\emptycomment{
\begin{defi}
A real {\bf K\"{a}hler} Leibniz algebra is a real pseudo-K\"{a}hler Leibniz algebra $(\huaE,\huaB,J)$, where the associated pseudo-Riemannian metric $S$ is positive definite??.
\end{defi}
}

At the end, we construct pseudo-K\"{a}hler Leibniz algebras using   Leibniz-dendriform algebras with   skew-symmetric invariant bilinear forms.

\begin{pro}
Let $(\huaL,\lhd,\rhd)$ be a real Leibniz-dendriform algebra with a skew-symmetric and nondegenerate invariant bilinear form $\omega$. Then $(\huaL\ltimes_{L^*_\lhd,-L^*_\lhd-R^*_\rhd}\huaL^*,\huaB,J)$ is a real pseudo-K\"{a}hler Leibniz algebra, where $J$ is given by \eqref{lei-den invariant to complex str} and $\huaB$ is given by \eqref{phase space symplectic}.
\end{pro}

\begin{proof}
By Theorem \ref{phase space and Lei-Den} and Proposition \ref{phase space complex product}, $\huaB$ is a symplectic structure and $J$ is a complex structure on the semidirect product Leibniz algebra $(\huaL\ltimes_{L^*_\lhd,-L^*_\lhd-R^*_\rhd}\huaL^*,[\cdot,\cdot]_\ltimes)$. For any $x,y\in\huaL$, $\xi,\eta\in\huaL^*$, we have
\begin{eqnarray*}
\huaB(J(x+\xi),J(y+\eta))&=&
\huaB(-\omega^{\sharp -1}(\xi)+\omega^{\sharp}(x),
-\omega^{\sharp -1}(\eta)+\omega^{\sharp}(y))\\
&=&-\langle\omega^{\sharp}(y),\omega^{\sharp -1}(\xi)\rangle
-\langle\omega^{\sharp}(x),\omega^{\sharp -1}(\eta)\rangle\\
&=&-\omega(y,\omega^{\sharp -1}(\xi))-\omega(x,\omega^{\sharp -1}(\eta))\\
&=&\omega(\omega^{\sharp -1}(\xi),y)+\omega(\omega^{\sharp -1}(\eta),x)\\
&=&\langle\xi,y\rangle+\langle\eta,x\rangle\\
&=&\huaB(x+\xi,y+\eta),
\end{eqnarray*}
which implies that $(\huaL\ltimes_{L^*_\lhd,-L^*_\lhd-R^*_\rhd}\huaL^*,\huaB,J)$ is a pseudo-K\"{a}hler Leibniz algebra.
\end{proof}

\noindent
{\bf Acknowledgements. } This research is supported by NSFC (11922110,12371029).

\noindent
{\bf Declaration of interests. } The authors have no conflicts of interest to disclose.

%The authors declare that they have no conflicts of interest to disclose.

\noindent
{\bf Data availability. } Data sharing is not applicable to this article as no new data were created or analyzed in this study.

 \end{document}